\documentclass[11pt,a4paper]{amsart}
\usepackage[utf8]{inputenc}
\usepackage[T1]{fontenc}

\pdfoutput=1

\title{Relative accessibility for graphs}
\author{Joseph Paul MacManus}
\date{12th May, 2026}

\address{School of Mathematics,  University of Bristol, Bristol, BS8 1UG, UK, and the Heilbronn Institute for Mathematical Research, Bristol, UK.}
\email{joseph.macmanus@bristol.ac.uk}

\usepackage{amsfonts}
\usepackage{amsmath}
\usepackage{amsthm}
\usepackage{amssymb}
\usepackage{stmaryrd}
\usepackage{mathrsfs}

\usepackage{graphicx}
\usepackage[dvipsnames]{xcolor}
\usepackage{tikz}
\usepackage{quiver}
\usepackage[colorlinks]{hyperref} 
\hypersetup{
    linkcolor=red,
    citecolor=blue,
}
\usepackage{fancyref}
\usepackage{todonotes}
\usepackage{mathabx}

\usepackage{thmtools, thm-restate}

\usepackage[abbrev,alphabetic]{amsrefs}

\DeclareMathOperator{\dHaus}{Haus}

\DeclareMathOperator{\Aut}{Aut}

\newcommand{\Z}{\mathbf{Z}}
\newcommand{\N}{\mathbf{N}}

\newcommand{\into}{\hookrightarrow}

\newcommand{\pres}[2]{\langle #1 \ ; \ #2 \rangle}

\newtheorem{theorem}{Theorem}[section]

\newtheorem*{theorem*}{Theorem}

\newtheorem{proposition}[theorem]{Proposition}
\newtheorem{lemma}[theorem]{Lemma}
\newtheorem{corollary}[theorem]{Corollary}
\newtheorem{claim}[theorem]{Claim}

\theoremstyle{definition}
\newtheorem{definition}[theorem]{Definition}
\newtheorem{example}[theorem]{Example}
\newtheorem{remark}[theorem]{Remark}

\newtheorem*{definition*}{Definition}

\newcommand{\per}{\mathcal{H}}
\newcommand{\br}{{\mathscr{B}}}

\newcommand{\cyc}{\mathscr{C}}

\makeatletter
\newcommand{\myitem}[1]{%
\item[#1]\protected@edef\@currentlabel{#1}%
}
\makeatother

\begin{document}

\begin{abstract}
    We relativise the Thomassen--Woess definition of accessibility in graphs, defining what it means for a graph to be \emph{accessible relative to a peripheral system}. In the case of locally finite, quasi-transitive graphs, we characterise relative accessibility in terms of a certain subring of the Boolean ring of the graph, and apply this to show that our definition agrees with the usual algebraic notion of relative accessibility in finitely generated groups. This implies, in particular, that relative accessibility is a quasi-isometry invariant amongst finitely generated groups, when the quasi-isometry coarsely preserves the left cosets of the peripheral subgroups.
    We also deduce a relative variant of Hamann's accessibility theorem on graphs with finitely generated cycle spaces. 

\end{abstract}

\maketitle




\section{Introduction}

The definition of an \emph{accessible group} is due to Wall \cite{wall1971pairs}. In the language of Bass--Serre theory, a finitely generated group $G$ is \emph{accessible} if it splits as a finite graph of groups with finite edge groups and vertex groups with at most one end. Wall conjectured that every finitely generated group is accessible, suggesting a natural extension of the Grushko--Neumann theorem. 
Much progress was made on this conjecture, most notably by Linnell \cite{linnell1983accessibility} and Dunwoody \cite{dunwoody1985accessibility}, who verified Wall's conjecture for groups with bounded-order finite subgroups and (almost) finitely presented groups, respectively. Eventually, a counterexample was discovered by Dunwoody in the 1993 \cite{dunwoody1993inaccessible}. 
That same year, the idea of accessibility was re-cast in an new, geometric light when Thomassen and Woess presented a purely graph-theoretical definition of accessibility \cite{thomassen1993vertex}. Their definition runs as follows.

\begin{definition}[Accessibility for graphs]\label{def:TW}
    Let $\Gamma$ be a connected, locally finite graph. We say that $\Gamma$ is (\emph{edge-})\emph{accessible} if there exists $k \in \N$ such that for every pair of distinct ends $\omega_1$, $\omega_2$ of $\Gamma$, there exists an edge-cut consisting of at most $k$ edges which separates $\omega_1$ and $\omega_2$. 
\end{definition}

Thomassen and Woess go on to prove that a finitely generated group is accessible if and only if some (and hence any) of its Cayley graphs are accessible graphs. This geometric characterisation has the immediate consequence that accessibility is a quasi-isometry invariant. Their work opened the door to questions about how the classical study of infinitely-ended groups can be applied more generally to quasi-transitive graphs through purely combinatorial means; see e.g. \cites{hamann2022stallings, moller1996accessibility, esperet2024structure, macmanus2023accessibility}.

Within geometric group theory, there is an established algebraic notion of `relative' accessibility, described in e.g. \cite{jsj}. Briefly, we are given a system of \emph{peripheral subgroups} along with our group, and we ask for a `maximal' finite splitting of the group where each peripheral subgroup fixes a vertex of the Bass--Serre tree; see Section~\ref{sec:groups} for a precise definition.
Studying properties of groups `relative to a given peripheral system' is a very successful idea, leading to many deep insights. The purpose of this paper is to advertise the following relativisation of Definition~\ref{def:TW}. 

\begin{definition}[Relative accessibility for graphs]\label{def:rel-acc-intro}
    Let $\Gamma$ be a connected, locally finite graph. A \emph{peripheral system} is a collection $\per$ of subsets of $V(\Gamma)$.
    We say that a cut $b \in \br(\Gamma)$ is \emph{$\per$-elliptic} if for every $H \in \per$,  either $b \cap H$ or $b^\ast \cap H$ is finite. 
    
    We then say that $\Gamma$ is \emph{accessible relative to $\per$} if there exists $k \geq 1$ such that for every pair of distinct ends $\omega_1$, $\omega_2$ of $\Gamma$, if there exists an $\per$-elliptic cut $b$ separating $\omega_1$ and $\omega_2$ then there exists a (possibly different) $\per$-elliptic cut $b'$ separating the $\omega_i$ such that $|\delta b'| \leq k$.
\end{definition}

 It can be verified (see Proposition~\ref{prop:subring}) that the set of $\per$-elliptic cuts forms a subring of the Boolean ring $\br(\Gamma)$, which we denote by $\br_\per(\Gamma)$. As is the case with absolute accessibility, we show that relative accessibility is characterised by the generation properties of this ring.

\begin{restatable}{alphtheorem}{fgchar}\label{thm:fg-char}
    Let $\Gamma$ be a connected, locally finite, quasi-transitive $G$-graph and $\per$ a thin, $G$-invariant peripheral system. The following are equivalent:
    \begin{enumerate}
        \item $\Gamma$ is accessible relative to $\per$.

        \item $\br_\per(\Gamma)$ is finitely generated as a $G$-module.
    \end{enumerate}
\end{restatable}

In fact, the finite generating set we obtain is nested. Here, \emph{thin} means that every vertex is contained in finitely many peripherals; see Definition~\ref{def:thin}. 
The proof of Theorem~\ref{thm:fg-char} passes through the `cone-off' of $\Gamma$ over $\per$, denoted $\widehat \Gamma_\per$; see Section~\ref{sec:cone-off}. The key observation is that---after possibly replacing $\per$ with something slightly more well-behaved---there is a canonical isomorphism
$$
\br_\per(\Gamma) \cong \br(\widehat \Gamma_\per).
$$
This allows access to the full strength of the Dicks--Dunwoody theory of the Boolean ring of cuts. Even so, care has to be taken in order to take advantage of this map, as na\"ively there is no control on how the cardinalities of coboundaries change when passing through the isomorphism. Regaining this control is therefore the crux of this problem, which is achieved by substituting $\per$ with yet another, even better-behaved peripheral system. Exactly what we mean by `better-behaved' is not immediately obvious here, and pinning this down constitutes much of the hard work done towards proving Theorem~\ref{thm:fg-char}.

Using Theorem~\ref{thm:fg-char}, we are able to prove that our definition of relative accessibility coincides with the standard algebraic definition for finitely generated groups. This motivates that our definition is, in some sense, the `correct' one.

\begin{restatable}{alphtheorem}{algaccchar}\label{thm:alg-vs-graph-relacc}
    Let $G$ be a finitely generated group and $\mathcal P$ a finite collection of subgroups. Let $\Gamma$ be some locally finite Cayley graph of $G$. Let 
    $$
    \per = \{gH : H \in \mathcal P, \ g \in G\}.
    $$
    Then $G$ is accessible relative to $\mathcal P$ if and only if $\Gamma$ is accessible relative to $\per$. 
\end{restatable}

Our proof of Theorem~\ref{thm:alg-vs-graph-relacc} is algebraic and combinatorial in nature, relating actions on trees with finite edge stabilisers to homomorphisms between Boolean rings.
We remark that the above theorem has the following immediate geometric consequence. 

\begin{restatable}{alphcor}{qiinv}\label{cor:qi-invariance}
    Let $G_1$, $G_2$ be finitely generated groups, equipped with fixed word metrics. Let $\mathcal P_i$ a finite collection of subgroups of $G_i$. Suppose that there exists a quasi-isometry $f : G_1 \to G_2$ 
    such that the following holds:
    \begin{enumerate}
        \item\label{itm:qi-1} For all $g \in G_1$, $H \in \mathcal P_1$, there exists $h \in G_{2}$, $K \in \mathcal P_{2}$ such that
    $$
    \dHaus(f(gH), hK)<\infty.
    $$ 

        \item\label{itm:qi-2} For all $h \in G_2$, $K \in \mathcal P_2$, there exists  $g \in G_1$, $H \in \mathcal P_1$ such that
    $$
    \dHaus(f(gH), hK)<\infty.
    $$ 
    \end{enumerate}
    Then, $G_1$ is accessible relative to $\mathcal P_1$ if and only if $G_2$ is accessible relative to $\mathcal P_2$. 
\end{restatable}

Note that we do not require the Hausdorff distances in the above statement to be uniformly bounded. 

As mentioned above, a cornerstone result in the theory of accessible groups is Dunwoody's accessibility theorem for (almost) finitely presented groups \cite{dunwoody1985accessibility}. The relative variant of this appears in \cite[Prop.~2.24]{jsj}. Hamann gave a combinatorial generalisation of Dunwoody's theorem, in terms of the cycle space of a quasi-transitive graph \cite{hamann2018accessibility}. Using our machinery, we are able to deduce the following relative variant of Hamann's theorem, by considering  the cycle space of the cone-off of our graph over its peripheral system.

\begin{restatable}{alphtheorem}{relcycs}\label{thm:rel-cyc-accessibility}
    Let $G$ be a connected, locally finite, quasi-transitive $G$-graph. Let $\per$ be a thin, $G$-invariant peripheral system. Suppose that ${\cyc(\widehat \Gamma_\per)}$ is topologically $G$-finitely generated. Then $\Gamma$ is accessible relative to $\per$. 
\end{restatable}

This recovers and generalises the fact that a finitely generated group $G$ which is finitely presented relative to some collection of peripheral subgroups $\mathcal P$ is accessible relative to $\mathcal P$. We remark that $\mathcal P$ is necessarily finite under this hypothesis \cite[Thm.~1.1]{osin2004relhyp}.

Note the condition of \emph{topological} $G$-finite generation in the above theorem. This is strictly weaker than asking ${\cyc(\widehat\Gamma_\per)}$ simply admit a $G$-finite generating set. The exact definition of topological generation and its distinction from the usual `algebraic' generation is discussed in Section~\ref{sec:top-gen} below. We remark that Hamann's theorem also holds verbatim when only assuming the weaker hypothesis of topological generation, with only a minor modification to the proof. This modification is explained in Appendix~\ref{app:hamann}. 

We are interested in this weakening as it is often more directly relevant to certain applications, most notably when dealing with planar graphs. Indeed, this has been observed several times in the literature \cites{dunwoody2007planar,groves1991remarks}. Perhaps most relevant for us, Theorem~\ref{thm:rel-cyc-accessibility} is applied by the present author with this weaker hypothesis in \cite{macmanus2023accessibility}, where it is used to prove that a locally finite, quasi-transitive graph which is quasi-isometric to a planar graph is accessible. Indeed, some of the techniques of the current paper originally appeared in a rudimentary form in an earlier draft of \cite{macmanus2023accessibility}. The purpose of this paper is to further explore and greatly elaborate on these techniques, as they appear to be independently deserving of study. 

\subsection*{Acknowledgements}

This work was supported by the Additional Funding Programme for Mathematical Sciences, delivered by EPSRC (EP/V521917/1) and the Heilbronn Institute for Mathematical Research.


\section{Preliminaries}

\subsection{Basic conventions}

Given a set $S$, we denote by $\mathscr P(S)$ its power set, and write $\mathscr F(S)$ for the set of finite or cofinite\footnote{By \textit{cofinite}, we mean subsets with finite complements.} subsets of $S$.

Given a graph $\Gamma$, we will always denote its vertex set by $V(\Gamma)$ and its edge set by $E(\Gamma)$. All graphs considered in this paper are simple, unweighted, and undirected. Given a group $G$, a graph $\Gamma$ is called a \emph{$G$-graph} if it is equipped with an $G$-action by automorphisms. If this action has finitely many orbits of vertices, we say that $\Gamma$ is a \emph{quasi-transitive $G$-graph}.  Given $H \subset V(\Gamma)$, write $\Gamma[H]$ for the subgraph induced by $H$; that is, $H$ together with all edges in $\Gamma$ whose endpoints both lie in $H$. 
Denote the length of the shortest path between two vertices $u,v \in V(\Gamma)$ by $d(u,v)$. Similarly, if $u \in V(\Gamma)$ and $H \subset V(\Gamma)$, write $d(u,H)$ for the length of the shortest path with one endpoint equal to $u$, and the other lying in $H$.
Given $r \geq 0$, write 
$$
B(H,r) = \{x \in V(\Gamma) : d(x, H) \leq r \}. 
$$
When discussing a $G$-action on some structure, we may denote by $G_x$ the stabiliser of some point/subset $x$. If there is any risk of ambiguity (e.g. between pointwise/setwise stabilisation), we will explicitly describe and clarify the stabiliser we are considering.

\subsection{Ends and limit sets} We now recall the basic theory of ends and limit sets in graphs. We will only recall this theory so far as we need it. To avoid any unnecessary pathologies, we only work with locally finite graphs when discussing ends. 

Let $\Gamma$ be a connected, locally finite graph. Let $P$, $Q$ be infinite, simple rays. We say that $P$ and $Q$ are \emph{end-equivalent} if for all finite subsets $K \subset V(\Gamma)$, we have that some tails $P$ and $Q$ are contained in the same component of $\Gamma \setminus K$. This is an equivalence relation on the set of simple rays. Equivalence classes are called \emph{ends} and we denote by $\Omega(\Gamma)$ the set of ends. There is a natural way to topologise $\Gamma \sqcup \Omega(\Gamma)$ as a compact, metrisable space, called the \emph{Freudenthal compactification of $\Gamma$}. We denote this compactification by $|\Gamma|$. The inclusion $\Gamma \into |\Gamma|$ is a topological embedding, and $\Omega(\Gamma)$ is a closed subset of $|\Gamma|$, and thus is naturally compact.  See \cite[\S21]{woess2000random} for details.

Given a subset $H \subset V(\Gamma)$, we denote by $\Lambda(H)$ its \emph{limit set} in $\Omega(\Gamma)$. That is, we define
$
\Lambda(H) := \overline H \cap \Omega(\Gamma)
$, 
where $\overline H$ denotes the closure of $H$ in $|\Gamma|$. If $G$ is a group acting on $\Gamma$, we write $\Lambda(G) := \Lambda(Gx)$, where $x \in V(\Gamma)$ is some choice of vertex. It is easy to see that this does not depend on the choice of $x$. In particular, we may even take finitely many vertices, and then it is a short exercise in point-set topology to verify that the equality
$
\Lambda(G \cdot \{x_1, \ldots, x_k\}) = \Lambda(G)
$
will hold.

We record the following classical facts about limit sets, which can be found in most standard texts on the subject; we will appeal to Woess' monograph \cite{woess2000random}.\footnote{The results we cite are often phrased in terms of \emph{contractive $G$-compactifications}; the Freudenthal compactification of a locally finite $G$-graph is an example of such a compactification \cite[Prop.~21.5]{woess2000random}.} 

\begin{proposition}[{\cite[Prop.~20.9(1)]{woess2000random}}]\label{prop:limsets-classification}
    Let $\Gamma$ be a connected, locally finite $G$-graph, then
    $$
    |\Lambda(G)| \in \{0,1,2,\infty\}.
    $$
\end{proposition}

Given $g \in G$, we may write $\Lambda(\langle g \rangle) =: \Lambda(g)$ for brevity. We may then classify any  $g \in G$ into one of three categories:
\begin{itemize}
    \item If $|\Lambda(g)| = 0$, then $g$ acts with bounded orbits on $\Gamma$ and is called \emph{elliptic}. 

    \item If |$\Lambda(g)| = 1$, then $g$ is called \emph{parabolic}. 

    \item If $|\Lambda( g)| = 2$, then $g$ is said to be \emph{loxodromic}.
\end{itemize}

We will only really be interested in loxodromics. The dynamics of their actions can be understood as follows.

\begin{proposition}[{\cite[Prop.~20.7(1)]{woess2000random}}]\label{prop:lox-action}
    Let $\Gamma$ be a connected, locally finite $G$-graph and $g \in G$ a loxodromic element, with say $\Lambda(g) = \{\lambda_+, \lambda_-\}$, where $g^{\pm n}x \to \lambda_{\pm}$ as $n \to \infty$. Then for all compact $W \subset \Omega(\Gamma) \setminus \{\lambda_\mp\}$, and all open neighbourhoods $U$ of $\lambda_\pm$, there exists $M > 0$ such that for all $n > M$ we have that $g^{\pm n}(W) \subset U$. 
\end{proposition}

There is plenty more to be said about the dynamics of the action of $G$ on $\Omega(\Gamma)$. However, all we really need is the following two results.

\begin{proposition}[{\cite[Prop.~20.9(2)]{woess2000random}}]\label{prop:independent-lox}
    Let $\Gamma$ be a connected, locally finite $G$-graph, such that $|\Lambda(G)| > 1$. Then $G$ contains a loxodromic element. 
    If $|\Lambda(G)| = \infty$, then there exists loxodromics $g, h \in G$ such that $\Lambda(g) \neq \Lambda(h)$. In particular, $|\Lambda(\langle g,h\rangle)| = \infty$.
\end{proposition}

Note that if $g, h \in G$ are two loxodromic elements such that $\Lambda(g) \neq \Lambda(h)$, then since $\Lambda(g) \cup \Lambda(h) \subset \Lambda(\langle g,h\rangle)$, we automatically have that $|\Lambda(\langle g,h\rangle)| = \infty$ by Proposition~\ref{prop:limsets-classification}.



\subsection{The Boolean ring of finite cuts}

Let $\Gamma$ be a connected graph. Given a subset $b \subset V(\Gamma)$, write $b^\ast = V(\Gamma) \setminus b$. Let $\delta b \subset E(\Gamma)$ denote the set of edges with one endpoint in $b$ and one in $b^\ast$. Borrowing terminology from algebraic topology, $\delta b$ is called the \emph{coboundary} of $b$. Let 
$$
\br (\Gamma) := \{b \in \mathscr P(V(\Gamma)), \ |\delta b| < \infty\}.
$$
This is a Boolean ring\footnote{Recall that a \emph{Boolean ring} is a ring where every element $x$ is \emph{idempotent}, i.e. satisfies $x^2 = x$.} where `addition' is symmetric difference and `multiplication' is intersection. 
We may abuse terminology and refer to an element of $\br(\Gamma)$ as a \emph{cut}. A cut $b \in \br(\Gamma)$ is said to be \emph{tight} if $\Gamma[b]$ and $\Gamma[b^\ast]$ are connected. 

If $\Gamma$ is locally finite then every $b \in \br(\Gamma)$ induces a clopen bipartition $\Omega(\Gamma) = W_1 \sqcup W_2$, where
$$
W_1 = \Lambda(b), \ \ W_2 = \Lambda(b^\ast).
$$
We say that $b$ \emph{separates} two ends $\omega_1, \omega_2 \in \Omega(\Gamma)$ if $\omega_1 \in W_1$, $\omega_2 \in W_2$, or vice versa. 

If $\Gamma$ is a $G$-graph then this induces an action of $G$ upon $\br (\Gamma)$. We may therefore view $\br (\Gamma)$ as a $G$-module by `forgetting' the multiplicative structure. 
Given $n \geq 0$, let
$
\br_n(\Gamma)
$
denote the subring of $\Gamma$ generated by elements $b$ such that $|\delta b| \leq n$. We will make use of the following remarkable theorem of Dicks and Dunwoody.

\begin{theorem}[{\cite[Thm.~II.2.20]{dicks1989groups}}]\label{thm:dicks-dunwoody}
    Let $\Gamma$ be a connected graph. Then there exists a chain 
    $$
    E_1 \subset E_2 \subset E_3 \subset \ldots,
    $$
    where $E_n \subset \br_n(\Gamma)$ is a nested, $\Aut(\Gamma)$-invariant subset consisting of tight elements, such that $E_n$ generates $\br_n(\Gamma)$ as a ring. 
    In particular, $E = \bigcup_n E_n$ is a nested, $\Aut(\Gamma)$-invariant set which generates $\br(\Gamma)$ as a ring. 
\end{theorem}

\begin{remark}
    Since the $E_n$ are nested, we actually have that $E_n$ generates $\br_n(\Gamma)$ additively, i.e. as an abelian group. See Proposition~\ref{prop:gen-set-upgrade} below. 
\end{remark}

When combined with the following standard fact, this result becomes particularly potent. 

\begin{proposition}[{\cite[Prop.~4.1]{thomassen1993vertex}}]
    Let $\Gamma$ be a connected graph, $k > 0$, and $e \in E(\Gamma)$. Then there exists only finitely many tight $b \in \br(\Gamma)$ such that $e \in \delta b$ and $|\delta b| \leq k$. 
\end{proposition}

In particular, this has the following consequence. 
\begin{corollary}\label{cor:dicks-dunwoody-gfinite}
    Let $\Gamma$ be a connected $G$-graph, such that $G$ acts on $E(\Gamma)$ with finitely many orbits. Then $\br_n(\Gamma)$ contains a nested, $G$-invariant, $G$-finite set $E_n$ which generates $\br_n(\Gamma)$ as an abelian group. In particular, $\br_n(\Gamma)$ is a finitely generated $G$-module. 
\end{corollary}

\subsection{Some more useful results}

Before we dive in, it will be helpful to record some basic observations for later reference. 
First, we have the following fact which was first observed by M\"oller \cite{mller1992ends}; see also \cite[Prop.~7.1]{thomassen1993vertex}.

\begin{proposition}\label{prop:separating-generators}
    Let $\Gamma$ be a connected, locally finite graph. Let $S \subset \br(\Gamma)$. Suppose $\omega_1, \omega_2 \in \Omega(\Gamma)$ are separated by some $b$ in the additive span of $S$. Then there exists $b' \in S$ which separates $\omega_1$ and $\omega_2$. 
\end{proposition}

\begin{proof}
    Let $a_1, \ldots, a_k \in S$ be such that $b = \sum_i a_i$. Suppose for the sake of a contradiction that no $a_i$ separates $\omega_1$ and $\omega_2$. Let $P_i$ be (the vertex set of) an infinite simple ray approaching $\omega_i$, such that for all $j$ we have that $P_i$ is contained entirely in either $a_j$ or $a_j^\ast$. Suppose that for each $j$, we have that either $P_1 \cup P_2 \subset a_j$ or $P_1 \cup P_2 \subset a_j^\ast$. Then $P_1$ and $P_2$ lie in the same piece of the Venn diagram formed by the $a_j$. In particular, we must either have that $P_1 \cup P_2 \subset b$ or $P_1 \cup P_2 \subset b^\ast$. This contradicts the fact that $b$ separates the $\omega_i$. 
\end{proof}

The next result provides a criterion for constructing ring generating sets of a given subring of $\br(\Gamma)$.

\begin{proposition}\label{prop:g-finite-criterion}
    Let $\Gamma$ be a connected, locally finite graph. Let $M \subset \br(\Gamma)$ be a subring. Let $S \subset M$ be such that
    $$
    \text{$\exists b \in M$ separating $\omega_1$, $\omega_2$ $\iff$ $\exists b' \in S$ separating $\omega_1$, $\omega_2$ }.
    $$
    Then for every $b \in M$ there exists $b' \in \langle S \rangle$, the subring generated by $S$, such that $b+b'$ is finite.
\end{proposition}

\begin{proof}
    Let $b \in M$ be arbitrary. We will find an element of $\langle S \rangle$ such that $b + b'$ is finite. Recall that $b$ induces a bipartition of $\Omega(\Gamma)$ into two clopen sets $W_1 \sqcup W_2$, say $W_1 = \Lambda(b)$ and $W_2 = \Lambda(b^\ast)$. Let us assume without loss of generality that both $W_1$ and $W_2$ are non-empty, lest we are done already. 

    Fix $p \in W_1$. For every $q \in W_2$, by assumption we have that there exists $b_q \in S$ such that $p \in \Lambda(b_q)$, $q \in \Lambda(b_q^\ast)$. Note that 
    $$
    W_1 \cup \bigcup_{q \in W_2} \Lambda(b_q)
    $$
    is an open cover of $\Omega(\Gamma)$. Since the space of ends is compact, there is a finite subcover. In particular, there exists a finite collection $q_1, \ldots, q_k$ such that $\bigcup_i \Lambda(b_i^\ast) \subset W_2$, where $b_i := b_{q_i}$. Thus, we have that $a_p := \bigcap_i b_i \in \langle S \rangle$ separates $p$ from all of $W_2$, i.e. $p \in \Lambda(a_p)$ and $W_2 \subset \Lambda(a_p^\ast)$.
    Now, we have an open cover 
    $$
    W_2 \cup \bigcup_{p \in W_1} \Lambda(a_p),
    $$
    where each $\Lambda(a_p)$ is disjoint from $W_2$. Passing to a finite subcover, we find a finite collection $p_1, \ldots, p_\ell$ such that $\bigcup_i \Lambda(a_i) = W_1$, where $a_i := a_{p_i}$. It follows that $c := \bigcap_i (a_i)^\ast$ induces the same clopen bipartition on $\Omega(\Gamma)$, up to orientation. More precisely, we have that $b + c^\ast$ is finite.
\end{proof}

As discussed earlier, we may sometimes wish to `forget' the multiplicative operation on a given subring of $\br(\Gamma)$. The generating sets we construct can often be salvaged during this erasure process, using the following neat observation; see the proof of \cite[Prop.~6.2]{thomassen1993vertex}. 

\begin{proposition}\label{prop:gen-set-upgrade}
    Let $\Gamma$ be a connected graph. Let $M \subset \br(\Gamma)$ be a subring and $S \subset M$ be closed under taking complements, such that $S$ generates $M$ as a ring. Suppose that $S$ is nested. Then $S$ generates $M$ as an abelian group.
\end{proposition}

\begin{proof}
    It is sufficient to observe that we may express intersection in terms of the symmetric difference operation, given nested elements. Suppose that $a,b \in S$, so $a$ and $b$ are nested. Then, we have that
    $$
    a\cap b = \begin{cases}
        \emptyset &\text{if $a \subset b^\ast$,}\\
        a &\text{if $a \subset b$,}\\
        b &\text{if $b \subset a$,}\\
        a + b + V(\Gamma) &\text{if $b^\ast \subset a$.}
    \end{cases}
    $$
    It is easy to see that these cases are exhaustive. Note that $V(\Gamma) = a + a^\ast$ for any $a \in S$, and so, since $S$ is closed under complementation, we also get this element for free in the additive span. The proposition follows.
\end{proof}


\section{Elliptic cuts}

In this section we introduce the notions of a \emph{peripheral system} and an \emph{elliptic cut}, and define what it means for a graph to be \emph{accessible relative to a peripheral system}.


\subsection{Peripheral systems and elliptic cuts}

We begin by introducing the following core definitions.

\begin{definition}[Peripheral system]
    Let $\Gamma$ be a connected graph. Then a subset $\per \subset \mathscr P(V(\Gamma))$ is called a \emph{peripheral system}. 
    If $\Gamma$ is a $G$-graph for some group $G$, and $\per$ satisfies that $g H \in \per$ for all $g \in G$, $H \in \per$, then $\per$ is called \emph{$G$-invariant}. 
\end{definition}

Note that if $\per$ is $G$-invariant then this induces an action of $G$ upon $\per$ by permutations.

\begin{definition}[Elliptic cuts]
    Let $\Gamma$ be a connected graph and $\per$ a peripheral system. A cut $b \in \br(\Gamma)$ is said to be \emph{$\per$-elliptic} if for all $H \in \per$, either $b \cap H$ is finite or $b^\ast \cap H$ is finite.
    Let 
    $\br_\per(\Gamma)$ denote the set of $\per$-elliptic cuts.
\end{definition}

We may often abuse terminology and refer to an $\per$-elliptic cut as simply an \emph{elliptic cut}, if $\per$ is clear from context.
The following is a straightforward exercise in the algebra of sets.

\begin{proposition}\label{prop:subring}
    Let $\Gamma$ be a connected graph and $\per$ a peripheral system. Then 
    $\br_\per (\Gamma)$ is a subring of $\br (\Gamma)$. 
\end{proposition}

\begin{proof}
    Let $b_1, b_2 \in \br_{\per}(\Gamma)$. We need to check that $b_1 \cap b_2$ and $b_1 + b_2$ are in $\br_{\per} (\Gamma)$. Fix $H \in \per$. We now split into three cases. 
    Suppose first that $b_1 \cap H$ and $b_2 \cap H$ are finite. Then $(b_1 \cap b_2) \cap H$ is certainly finite and so $b_1 \cap b_2 \in \br_\per (\Gamma)$. Also, 
    $$
    (b_1 + b_2) \cap H = b_1 \cap H + b_2 \cap H,
    $$
    which is finite.
    Secondly, assume now that $b_1 \cap H$  and $b_2 ^\ast \cap H$ are finite. Then we see that $(b_1 \cap b_2) \cap H \subset b_1 \cap H$ is finite. Note that
    $$
    (b_1 + b_2)^\ast \cap H= (H \cap b_1 ^\ast \cap b_2^\ast) \cup (H \cap b_1 \cap b_2), 
    $$
    and so $(b_1 + b_2)^\ast \cap H$ is also finite.
    Finally, in the third case we suppose that $b_1^\ast \cap H$ and $b_2 ^\ast \cap H$ are finite. Then 
    $$
    H \cap (b_1 \cap b_2)^\ast = H \cap (b_1 ^\ast \cup b_2 ^\ast ) = (b_1^\ast \cap H) \cup (b_2^\ast \cap H),
    $$
    which is finite. Furthermore, $b_1^\ast + b_2^\ast = b_1 + b_2$, and so 
    $$
    (b_1 + b_2) \cap H = (b_1 ^\ast + b_2 ^\ast ) \cap H = b_1^\ast \cap H + b_2 ^\ast \cap H, 
    $$
    which is finite. 
    In each case, we have shown that both $b_1 \cap b_2$ and $b_1 + b_2$ lie in $\br_\per (\Gamma)$. It follows that $\br_\per (\Gamma)$ is a subring of $\br (\Gamma)$. 
\end{proof}

We will almost exclusively be interested in peripheral systems which satisfy the following natural hypothesis.

\begin{definition}[Thin peripheral system]\label{def:thin}
    Let $\Gamma$ be a connected graph, then a peripheral system $\per$ is called \emph{thin} if for all $v \in V(\Gamma)$, there exists at most finitely many $H \in \per$ such that $v \in H$.
\end{definition}

Note that if $\Gamma$ is a connected, locally finite, quasi-transitive $G$-graph and $\per$ is thin and $G$-invariant, then $\per/G$ is finite. 

Our standing assumption on any peripheral system will usually be that it is thin. However, thinness alone is often not enough to immediately say anything meaningful, as thin peripheral systems can still be wild.
Much of the hard work done in this section is to establish how we can take a thin peripheral system $\per$ and substitute in its place a `nicer' system $\per'$. What `nicer' here means is not obvious, and we will pass through several iterations of `niceness' before arriving at something well-behaved enough that we gain some useful geometric control. The pay-off of all this is that, in the end, our substitutions ultimately have no effect on the material content of our results, and our main theorems will still only ask that the hypothesised peripheral system be thin. 

The first iteration of `niceness' we will consider is the following strengthening of the thin property, which we call \emph{tameness}. 

\begin{definition}[Tame peripheral system]
    Let $\Gamma$ be a connected graph, then a peripheral system $\per$ is called \emph{tame} if for any elliptic cut $b \in \br _\per (\Gamma)$ we have that only finitely many $H \in \per$ intersect both $b$ and $b^\ast$. 
\end{definition}

It is easy to see that every tame peripheral system of a locally finite graph is thin, in the above sense. The converse is false, even for quasi-transitive $G$-graphs with $G$-invariant peripheral systems, as can be seen in Example~\ref{exa:not-tame} below. 

\begin{example}\label{exa:not-tame}
    Consider the 3-regular tree $\Gamma = T_3$ with a distinguished end $\omega_0 \in \Omega(\Gamma)$, thus viewing $\Gamma$ as an infinite binary tree.
    Let $G \subset \Aut(\Gamma)$ be the subgroup which stabilises $\omega_0$, and thus preserves the level sets. It is easy to see that $G$ acts on $\Gamma$ vertex- and edge-transitively. 
    Let $\per$ denote the set of level sets, which is $G$-invariant and thin. Moreover, it is not too hard to see that every $b \in \br(\Gamma)$ is an elliptic cut, and so $\br_\per(\Gamma) = \br(\Gamma)$. However, $\per$ is not tame, as removing any edge from $\Gamma$ separates vertices in infinitely many level sets. See Figure~\ref{fig:nontameexample} for a cartoon.
    \begin{figure}[h]
         \centering
        \input{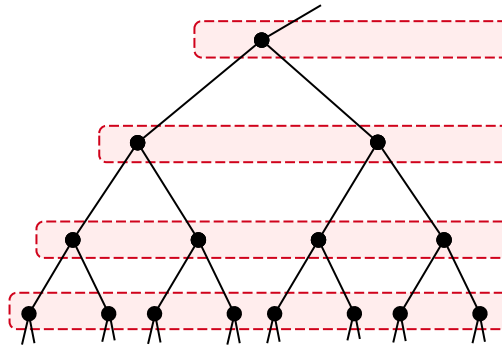}
        \caption{Example of a  thin peripheral system which is not tame.}
        \label{fig:nontameexample}
    \end{figure}
\end{example}

\subsection{Coarse connectivity of peripherals}

We now take a quick detour to introduce some terminology to help us understand the geometric structure of peripherals. In particular, we will look at their `coarse components'. 

\begin{definition}[Coarse components]
    Let $\Gamma$ be a connected graph, $H \subset V(\Gamma)$, and $r > 0$. We say that two vertices $x,y\in H$ \emph{lie in the same $r$-coarse component of $H$} if they are contained in the same connected components of the induced graph $\Gamma[B(H,r)]$. 
\end{definition}

Note that the $r$-coarse components of $H$ naturally partition $H$, much in the same way that connected components do.
For brevity, we introduce the following terminology.

\begin{definition}
    We say that $H$ has \emph{unbounded coarse components} if there exists $r > 0$ such that every $r$-coarse component of $H$ is infinite in diameter. 

    If there exists $r > 0$ such that every $r$-coarse component $K$ of $H$ satisfies $|\Lambda(K)| = \infty$, then we say that $H$ has \emph{big coarse components}. 

    Finally, if there exists $r > 0$ such that $H$ has exactly one $r$-coarse component, then we say that $H$ is \emph{($r$-)coarsely connected}. 
\end{definition}

\begin{example}
    The peripheral system in Example~\ref{exa:not-tame} does not have unbounded coarse components, despite every peripheral containing infinitely many vertices. 
\end{example}




\begin{remark}\label{rem:replace-with-neighbourhoods}
    It is easy to see that, given $r > 0$, if we replace a peripheral system $\per$ with
    $$
    \per_{+r} := \{B(H,r) : H \in \per\}, 
    $$
    then $\br_\per(\Gamma) = \br_{\per_{+r}}(\Gamma)$. Moreover, if $\per$ is thin/tame, then so is $\per_{+r}$. 
    %
\end{remark}

\subsection{Minimised peripheral systems}

There are many ways a peripheral system can contain `redundancies'. The first, most obvious one is that some peripherals may be finite. 
Note that we do not, in general, lose anything by assuming that every $H \in \per$ is infinite. Given a peripheral system $\per$, we can define
$$
\per_{\inf} := \{H \in \per : |H| = \infty\}.
$$
Then, it is easy to see that any $b \in \br (\Gamma)$ is $\per$-elliptic if and only if it is $\per_{\inf}$-elliptic.
We can do slightly better than this by inspecting the limit sets of peripherals.

\begin{definition}[Minimised peripheral system]
    Let $\Gamma$ be a connected, locally finite graph and $\per$ a peripheral system. The \emph{minimisation} of $\per$, denoted $\per^\bullet$, is defined as
    $$
    \per^\bullet  := \Big\{ H \in \per: |\Lambda(H)| > 1\Big\}.
    $$
    We say that $\per$ is \emph{minimised} if $\per = \per^\bullet$. 
\end{definition}

\begin{proposition}\label{prop:min-preserves-cuts}
    Let $\Gamma$ be a connected, locally finite graph and $\per$ a peripheral system. Then $\br_\per(\Gamma) = \br_{\per^\bullet}(\Gamma)$. 
\end{proposition}

\begin{proof}
    Fix $b \in \br (\Gamma)$. Firstly, it is easy to see that if b is $\per$-elliptic then $b$ is $\per^\bullet$-elliptic, since $\per^\bullet \subset \per$. 
    Conversely, suppose that $b$ is $\per^\bullet$-elliptic. Given $H \in \per \setminus \per^\bullet$, we have that $|\Lambda(H)| \leq 1$, which implies that either $b \cap H$ is finite or $b^\ast \cap H$ is finite. In particular, $b$ is $\per$-elliptic. The proposition follows.
\end{proof}

The following justifies our decision to minimise our peripheral sets.




    


\begin{lemma}\label{lem:unbounded-coarse-comps}
    Let $\Gamma$ be a connected, locally finite, quasi-transitive $G$-graph. Let $\per$ be a minimised, $G$-invariant, thin peripheral system. Then the following hold:
    \begin{enumerate}
        \item\label{itm:unbounded-coarsecomps}  Every $H \in \per$ has unbounded coarse components.

        \item\label{itm:big-coarse-comps} If $H \in \per$ satisfies $|\Lambda(H)| = \infty$, then $H$ has big coarse components.
    \end{enumerate}
\end{lemma}

\begin{proof}
    Fix $H \in \per$. By Proposition~\ref{prop:independent-lox}, there exists a loxodromic element $t \in G_H$, where $G_H \leq G$ denotes the setwise stabiliser of $H$. Thus, for every $x \in H$, there exists $r > 0$ such that $\Gamma[B(\langle t\rangle \cdot x,r]$ is connected. Since $G_H$ acts on $H$ with finitely many orbits, we may take this value of $r$ uniformly. In particular, $H$ has unbounded coarse components. 

    Suppose the limit set of $H$ is infinite. Then by Proposition~\ref{prop:independent-lox} there exists a finitely generated subgroup $K \subset G_H$ such that $\Lambda(K)$ is infinite. In particular, $K$-orbits of vertices are coarsely connected. Again, since $G_H$ acts with finitely many orbits on $H$, we deduce that $H$ has big coarse components. 
\end{proof}

\begin{lemma}\label{lem:unbounded-components-tame}
    Let $\Gamma$ be a connected, locally finite, quasi-transitive $G$-graph. Let $\per$ be a $G$-invariant, thin peripheral system. Suppose every infinite $H \in \per$ has unbounded coarse components. Then $\per$ is tame. 
\end{lemma}

\begin{proof}
    Suppose $\per$ is not tame, so there exists $b \in \br_\per(\Gamma)$ and an infinite sequence $H_1, H_2 , \ldots \in \per$ of distinct elements such that both $H_i \cap b$ and $H_i \cap b^\ast$ are non-empty for all $i \geq 1$. Without loss of generality, let us assume that each $H_i$ is infinite (indeed, only finitely many of the $H_i$ can be finite since $\per$ is thin), and $b^\ast \cap H_i$ is finite for every $i \geq 1$. We have $H_i$ has unbounded $r$-coarse components, for some uniform $r > 0$ since $\per /G$ is finite. It follows every $H_i$ intersects a bounded neighbourhood around $\delta b$, and so $\per$ cannot be thin. 
\end{proof}

The above lemmas have the following nice consequence. 

\begin{theorem}\label{lem:thin-min-implies-tame}
    Let $\Gamma$ be a connected, locally finite, quasi-transitive $G$-graph. Let $\per$ be a minimised, thin, $G$-invariant peripheral system. Then $\per$ is tame.
\end{theorem}

\begin{proof}
    By Lemma~\ref{lem:unbounded-coarse-comps}(\ref{itm:unbounded-coarsecomps}), since $\per$ is minimised we have that each $H \in \per$ is infinite and has unbounded coarse components, and so we conclude that $\per$ is tame by Lemma~\ref{lem:unbounded-components-tame}.
\end{proof}

\subsection{Separable peripheral systems}

Minimisation is not the only way we can simplify the structure of our peripheral systems. It is quite common that two peripherals might be completely indistinguishable by elliptic cuts, in that they always fall on the same side of any such cut. When this happens, we might as well combine these peripherals into one. This is formalised by the following definition.

\begin{definition}[Separable peripheral system]
    Let $\Gamma$ be a connected graph and $\per$ a peripheral system. Given $H_1, H_2 \in \per$, we say that $H_1$ and $H_2$ are \emph{distinguishable} if there exists $b \in \br_\per(\Gamma)$ such that $b \cap H_1$ and $b^\ast \cap H_2$ are both infinite. If $H_1$ and $H_2$ are not distinguishable, we say they are \emph{indistinguishable}. 
    
    We say that $\per$ is \emph{separable} if all pairs of distinct, infinite $H_1, H_2 \in \per$ are distinguishable.
\end{definition}

The following proposition is easy, but useful to note.

\begin{proposition}\label{prop:distinguishable-disjointlimset}
    Let $\Gamma$ be a connected, locally finite graph, $\per$ be a peripheral system, and $H_1, H_2 \in \per$. 
    Suppose $H_1$ and $H_2$ are distinguishable. Then $\Lambda(H_1) \cap \Lambda(H_2) = \emptyset$. 
\end{proposition}

\begin{proof}
    Suppose there exists $\omega \in\Lambda(H_1) \cap \Lambda(H_2) $, and also that there exists $b \in \br_\per (\Gamma)$ such that $b \cap H_1$ and $b^\ast \cap H_2$ are both infinite. Any sequence of points in $\Gamma$ which approaches $\omega$ must eventually  be contained in one of $b$ or $b ^\ast$, say $b$ without loss of generality. But then $H_1$ and $H_2$ both have infinite intersection with $b$, and so in particular $b$ cannot be elliptic. 
\end{proof}

Given a non-separable peripheral system $\per$, we can define a relation $\sim$ on $\per$, where 
$$
H_1 \sim H_2 \ \ \Longleftrightarrow \ \ 
\begin{cases}
    \text{$H_1 = H_2$, or}\\ \text{$H_1$ and $H_2$ are both infinite and indistinguishable.}
\end{cases}
$$
We now verify this is an equivalence relation.

\begin{proposition}
    Let $\Gamma$ be a connected graph and $\per$ a peripheral system. Then $\sim$ is an equivalence relation on $\per$. 
\end{proposition}

\begin{proof}
    Firstly, $\sim$ is obviously reflexive and symmetric. We need only verify transitivity. Suppose $H_1 \sim H_2$ and $H_2 \sim H_3$, but there exists $b \in \br _\per(\Gamma)$ which distinguishes $H_1$ and $H_3$. Then $H_2$ must have infinite intersection with either $b$ or $b^\ast$. But then $b$ must distinguish $H_2$ from either $H_1$ or $H_3$, a contradiction. 
\end{proof}

This allows us to define the \emph{consolidation} of a peripheral system $\per$, which we denote as $\per^\circ$, via
$$
\per^{\circ} := \Bigg\{ \bigcup_{H \in C} H : C \in \per /\sim\Bigg\}.
$$
Note that na\"ive consolidation may sometimes change the set of elliptic cuts, as is illustrated by the following examples. 

\begin{example}\label{exa:typei-reduction-changes-cuts}
    Consider the following examples:
    \begin{enumerate}
        \item Let $\Gamma$ be a bi-infinite path with $V(\Gamma) = \Z$ identified in the natural way. Let
    $$
    \per = \Big\{ \{m :  m> n\} : n \in \Z \Big\}.
    $$
    Then $\br_\per(\Gamma) = \br(\Gamma)$. However, $\per ^\circ = \{V(\Gamma)\}$, and so it follows that $\br_{\per^\circ}(\Gamma) = \mathscr F (V(\Gamma))$. 

    \item For a thin example, we may return our gaze to Example~\ref{exa:not-tame}. We see in this example that $\per^\circ = \{V(\Gamma)\}$, and thus once again we conclude that
    $$
    \br_\per(\Gamma) = \br(\Gamma) \neq  \mathscr F (V(\Gamma))= \br_{\per^\circ}(\Gamma).
    $$
    \end{enumerate}

\end{example}

For balance, we present the following interesting example of consolidation of peripheral systems in planar graphs, which demonstrates how consolidation can return a much more `natural' peripheral system.

\begin{example}
    Let $\Gamma$ be an infinite-ended, 2-connected, locally finite, planar graph, with some fixed embedding $\vartheta : |\Gamma| \to \mathbb S^2$ of its Freudenthal compactification in the 2-sphere. This embedding exists by \cite[Lem.~12]{richter20023}. Each complementary components of the image is a topological disk, and its boundary is a simple closed curve \cite[Prop.~3]{richter20023}. These simple closed curves maybe correspond to either finite facial cycles in $\Gamma$, or possibly infinite topological circuits, called the \emph{infinite facial circuits}. In general, an infinite facial circuit is formed of a disjoint union of bi-infinite simple rays, called \emph{facial rays}, together with a collection of ends. 

    Consider the peripheral system $\mathcal F = \{V(R) : \text{$R$ is a facial ray}\}$. It is not too hard to see that two facial rays $R_1, R_2 \in \mathcal F$ are distinguishable if and only if they do not lie on a common facial circuit. Indeed, this follows from cut-cycle duality in planar graphs. In particular, we have that 
    $$
    \mathcal F^\circ = \{V(C) : \text{$C \subset |\Gamma|$ is a facial circuit}\}. 
    $$
    Thus, consolidation has recovered a very natural partition of our peripherals. See Figure~\ref{fig:planar-consol-example} for a cartoon.
\end{example}

\begin{figure}
     \centering
    \tikzset{every picture/.style={line width=0.75pt}} 

\begin{tikzpicture}[x=0.75pt,y=0.75pt,yscale=-1,xscale=1]

\draw  [color={rgb, 255:red, 74; green, 144; blue, 226 }  ,draw opacity=1 ][fill={rgb, 255:red, 231; green, 241; blue, 255 }  ,fill opacity=1 ][dash pattern={on 3.75pt off 1.5pt}] (412.5,130.17) .. controls (412.5,89.21) and (445.71,56) .. (486.67,56) .. controls (527.63,56) and (560.83,89.21) .. (560.83,130.17) .. controls (560.83,171.13) and (527.63,204.33) .. (486.67,204.33) .. controls (445.71,204.33) and (412.5,171.13) .. (412.5,130.17) -- cycle ;
\draw  [color={rgb, 255:red, 74; green, 144; blue, 226 }  ,draw opacity=1 ][fill={rgb, 255:red, 255; green, 255; blue, 255 }  ,fill opacity=1 ][dash pattern={on 3.75pt off 1.5pt}] (425.46,130.17) .. controls (425.46,96.36) and (452.86,68.96) .. (486.67,68.96) .. controls (520.47,68.96) and (547.88,96.36) .. (547.88,130.17) .. controls (547.88,163.97) and (520.47,191.38) .. (486.67,191.38) .. controls (452.86,191.38) and (425.46,163.97) .. (425.46,130.17) -- cycle ;
\draw  [color={rgb, 255:red, 208; green, 2; blue, 27 }  ,draw opacity=1 ][fill={rgb, 255:red, 255; green, 216; blue, 221 }  ,fill opacity=1 ][dash pattern={on 3.75pt off 2.25pt}] (347.33,161.75) .. controls (357.54,149.1) and (356.33,136.25) .. (355.58,124.75) .. controls (354.83,113.25) and (346.95,99.11) .. (334.17,87) .. controls (321.38,74.89) and (305.7,71.81) .. (305.58,70) .. controls (305.46,68.19) and (304.08,68.5) .. (305.08,65) .. controls (306.08,61.5) and (307.58,62.5) .. (310.08,62.5) .. controls (312.58,62.5) and (344.08,74.75) .. (352.83,90.25) .. controls (361.58,105.75) and (363.83,106.75) .. (364.58,124) .. controls (365.33,141.25) and (365.58,149.5) .. (357.58,164.75) .. controls (349.58,180) and (336.83,193.75) .. (333.58,193) .. controls (330.33,192.25) and (330.41,190.9) .. (330,190.92) .. controls (329.59,190.93) and (328.08,189.25) .. (327.83,186.25) .. controls (327.58,183.25) and (337.13,174.4) .. (347.33,161.75) -- cycle ;
\draw  [color={rgb, 255:red, 208; green, 2; blue, 27 }  ,draw opacity=1 ][fill={rgb, 255:red, 255; green, 216; blue, 221 }  ,fill opacity=1 ][dash pattern={on 3.75pt off 2.25pt}] (284.42,205.25) .. controls (268.42,203) and (257.84,195.58) .. (252.42,191.5) .. controls (246.99,187.42) and (238.98,178.65) .. (238.42,177.25) .. controls (237.85,175.85) and (238.42,172.25) .. (239.92,171.5) .. controls (241.42,170.75) and (243.92,169.25) .. (246.17,171.75) .. controls (248.42,174.25) and (262.67,187.25) .. (269.92,190.75) .. controls (277.17,194.25) and (284,195.19) .. (292.92,195.25) .. controls (301.83,195.31) and (314.35,192.25) .. (317.17,192.25) .. controls (319.99,192.25) and (320.17,192.5) .. (321.42,195.5) .. controls (322.67,198.5) and (321.17,198) .. (319.42,199) .. controls (317.67,200) and (300.42,207.5) .. (284.42,205.25) -- cycle ;
\draw  [color={rgb, 255:red, 208; green, 2; blue, 27 }  ,draw opacity=1 ][fill={rgb, 255:red, 255; green, 216; blue, 221 }  ,fill opacity=1 ][dash pattern={on 3.75pt off 2.25pt}] (225.42,153.25) .. controls (224.42,149.75) and (222.67,147.25) .. (223.67,145.25) .. controls (224.67,143.25) and (225.11,142.96) .. (227.58,143.25) .. controls (230.06,143.54) and (228.67,143.25) .. (229.92,144.75) .. controls (231.17,146.25) and (231.42,149.5) .. (232.42,152.25) .. controls (233.42,155) and (238.17,159.75) .. (236.67,162.5) .. controls (235.17,165.25) and (237.17,164.5) .. (234.67,165.5) .. controls (232.17,166.5) and (232.77,165.37) .. (229.92,162) .. controls (227.07,158.63) and (226.42,156.75) .. (225.42,153.25) -- cycle ;
\draw  [color={rgb, 255:red, 208; green, 2; blue, 27 }  ,draw opacity=1 ][fill={rgb, 255:red, 255; green, 216; blue, 221 }  ,fill opacity=1 ][dash pattern={on 3.75pt off 2.25pt}] (235.08,85.5) .. controls (245.29,72.85) and (251.47,71.76) .. (260.83,67.25) .. controls (270.2,62.74) and (284.61,59.46) .. (287.08,59.75) .. controls (289.56,60.04) and (290.9,62.86) .. (290.83,64) .. controls (290.76,65.14) and (291.08,65) .. (289.83,68.75) .. controls (288.58,72.5) and (264.08,73.25) .. (250.58,88.25) .. controls (237.08,103.25) and (232.58,126) .. (231.08,128.75) .. controls (229.58,131.5) and (226.91,132.15) .. (226.5,132.17) .. controls (226.09,132.18) and (225.43,134.12) .. (222.58,130.75) .. controls (219.73,127.38) and (224.88,98.15) .. (235.08,85.5) -- cycle ;
\draw  [dash pattern={on 0.84pt off 2.51pt}] (226.5,132.17) .. controls (226.5,95.07) and (256.57,65) .. (293.67,65) .. controls (330.76,65) and (360.83,95.07) .. (360.83,132.17) .. controls (360.83,169.26) and (330.76,199.33) .. (293.67,199.33) .. controls (256.57,199.33) and (226.5,169.26) .. (226.5,132.17) -- cycle ;
\draw  [draw opacity=0][line width=1.5]  (316.74,195.55) .. controls (309.54,198.17) and (301.77,199.6) .. (293.67,199.6) .. controls (272.9,199.6) and (254.32,190.22) .. (241.95,175.45) -- (293.67,132.17) -- cycle ; \draw  [line width=1.5]  (316.74,195.55) .. controls (309.54,198.17) and (301.77,199.6) .. (293.67,199.6) .. controls (272.9,199.6) and (254.32,190.22) .. (241.95,175.45) ;  
\draw  [draw opacity=0][line width=1.5]  (226.4,127.27) .. controls (228.75,94.62) and (254.34,68.39) .. (286.72,65.08) -- (293.67,132.17) -- cycle ; \draw  [line width=1.5]  (226.4,127.27) .. controls (228.75,94.62) and (254.34,68.39) .. (286.72,65.08) ;  
\draw  [draw opacity=0][line width=1.5]  (307.21,66.09) .. controls (337.96,72.36) and (361.1,99.56) .. (361.1,132.17) .. controls (361.1,155.39) and (349.37,175.86) .. (331.51,187.99) -- (293.67,132.17) -- cycle ; \draw  [line width=1.5]  (307.21,66.09) .. controls (337.96,72.36) and (361.1,99.56) .. (361.1,132.17) .. controls (361.1,155.39) and (349.37,175.86) .. (331.51,187.99) ;  
\draw  [draw opacity=0][line width=1.5]  (233.47,162.6) .. controls (230.81,157.35) and (228.82,151.7) .. (227.6,145.76) -- (293.67,132.17) -- cycle ; \draw  [line width=1.5]  (233.47,162.6) .. controls (230.81,157.35) and (228.82,151.7) .. (227.6,145.76) ;  
\draw  [dash pattern={on 0.84pt off 2.51pt}] (419.5,130.17) .. controls (419.5,93.07) and (449.57,63) .. (486.67,63) .. controls (523.76,63) and (553.83,93.07) .. (553.83,130.17) .. controls (553.83,167.26) and (523.76,197.33) .. (486.67,197.33) .. controls (449.57,197.33) and (419.5,167.26) .. (419.5,130.17) -- cycle ;
\draw  [draw opacity=0][line width=1.5]  (509.74,193.55) .. controls (502.54,196.17) and (494.77,197.6) .. (486.67,197.6) .. controls (465.9,197.6) and (447.32,188.22) .. (434.95,173.45) -- (486.67,130.17) -- cycle ; \draw  [line width=1.5]  (509.74,193.55) .. controls (502.54,196.17) and (494.77,197.6) .. (486.67,197.6) .. controls (465.9,197.6) and (447.32,188.22) .. (434.95,173.45) ;  
\draw  [draw opacity=0][line width=1.5]  (419.4,125.27) .. controls (421.75,92.62) and (447.34,66.39) .. (479.72,63.08) -- (486.67,130.17) -- cycle ; \draw  [line width=1.5]  (419.4,125.27) .. controls (421.75,92.62) and (447.34,66.39) .. (479.72,63.08) ;  
\draw  [draw opacity=0][line width=1.5]  (500.21,64.09) .. controls (530.96,70.36) and (554.1,97.56) .. (554.1,130.17) .. controls (554.1,153.39) and (542.37,173.86) .. (524.51,185.99) -- (486.67,130.17) -- cycle ; \draw  [line width=1.5]  (500.21,64.09) .. controls (530.96,70.36) and (554.1,97.56) .. (554.1,130.17) .. controls (554.1,153.39) and (542.37,173.86) .. (524.51,185.99) ;  
\draw  [draw opacity=0][line width=1.5]  (426.47,160.6) .. controls (423.81,155.35) and (421.82,149.7) .. (420.6,143.76) -- (486.67,130.17) -- cycle ; \draw  [line width=1.5]  (426.47,160.6) .. controls (423.81,155.35) and (421.82,149.7) .. (420.6,143.76) ;  

\draw (240.25,115.4) node [anchor=north west][inner sep=0.75pt]  [color={rgb, 255:red, 208; green, 2; blue, 27 }  ,opacity=1 ]  {$\mathcal{F}$};
\draw (432.25,110.4) node [anchor=north west][inner sep=0.75pt]  [color={rgb, 255:red, 74; green, 144; blue, 226 }  ,opacity=1 ]  {$\mathcal{F}^{\circ }$};
\draw (371.33,121.9) node [anchor=north west][inner sep=0.75pt]    {$\Longrightarrow $};

\end{tikzpicture}
    \caption{The effect of consolidation on the facial rays of a planar graph.}
    \label{fig:planar-consol-example}
\end{figure}
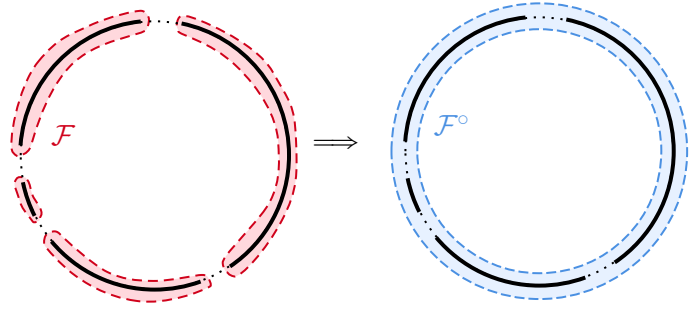

The above examples show that we must certainly take some care while consolidating our peripheral systems, but doing so can really help illuminate the relationship between our peripherals.
To clarify the situation seen in the bad examples of Example~\ref{exa:typei-reduction-changes-cuts}, we prove the following. 

\begin{proposition}\label{prop:consol-preserves-cuts-when-tame}
    Let $\Gamma$ be a connected graph, and $\per$ a peripheral system. Then $\br_{\per^\circ}(\Gamma) \subseteq \br_\per(\Gamma)$. If $\per$ is tame, then we have equality.
\end{proposition}

\begin{proof}
    Suppose that $b$ is $\per^\circ$-elliptic. Then it is easy to see that $b$ is $\per$-elliptic, since elements of $\per^\circ$ are just unions of elements of $\per$. This proves the first part of the proposition.
    
    Assume further that $\per$ is tame, and suppose that $b$ is $\per$-elliptic, but there exists $H \in \per^\circ$ such that both $b \cap H$ and $b^\ast \cap H$ are infinite. Write 
    $
    H = \bigcup_{H' \in C} H'
    $, 
    where $C$ is some $\sim$-equivalence class in $\per$. Since $b$ is $\per$-elliptic and $\per$ is tame, there must exist $H_1, H_2 \in C$ such that $H_1 \cap b$ and $H_2 \cap b^\ast$ are both infinite. But then $H_1$ and $H_2$ are infinite and distinguishable, which contradicts the assumption that $H_1 \sim H_2$. This proves the proposition.
\end{proof}

We now verify that consolidation preserves tameness. 

\begin{proposition}\label{prop:tameness-preserved}
    Let $\Gamma$ be a connected graph and $\per$ a peripheral system. Suppose $\per$ is tame. Then $\per^\circ$ is tame. 
\end{proposition}

\begin{proof}
    By Proposition~\ref{prop:consol-preserves-cuts-when-tame}, since $\per$ is tame we know that 
    $
    \br_{\per^\circ}(\Gamma) = \br_{\per}(\Gamma)
    $. 
    Let $b \in \br_\per(\Gamma)$. 
    Suppose then that there exists infinitely many $H \in \per^\circ$ which intersect both $b$ and $b^\ast$, say $H_1, H_2, \ldots \in \per^\circ$. Without loss of generality, we may assume that each $H_i$ is a union
    $$
    H_i = \bigcup_j H_i^j
    $$
    of at least two distinct, infinite, pairwise indistinguishable elements of $\per$. Since $\br_{\per^\circ}(\Gamma) = \br_{\per}(\Gamma)$, we deduce that for each $i$ there is some $j$ such that $H_i^j$ intersects both $b$ and $b^\ast$. Let $H_i'$ denote such a choice. We necessarily have that $H_i ' \neq H_j'$ for all $i \neq j$, but then $\per$ itself is not tame, a contradiction.
\end{proof}

Of course, we must verify that consolidation actually produces separable peripheral systems. This is the content of the next proposition. 

\begin{proposition}\label{prop:actually-seperable}
    Let $\Gamma$ be a connected graph and $\per$ a tame peripheral system. Then $\per^\circ$ is separable.
\end{proposition}

\begin{proof}
    We may assume without loss of generality that $\per^\circ$ contains at least two infinite elements, lest there is nothing to prove.
    Let $H_1, H_2 \in \per^\circ$ be distinct, infinite elements, say
    $
    H_i = \bigcup_{K \in C_i} K
    $
    for some some distinct equivalence classes $C_1, C_2 \in \per/\sim$. Let $K_i \in C_i$ be arbitrary. Note that the $K_i$ are necessarily infinite. Since $K_1$ and $K_2$ are not in the same equivalence class, they are distinguishable. Thus, there exists some $b \in \br_\per(\Gamma)$ such that $K_1 \cap b$ is infinite and $K_2 \cap b^\ast$ is infinite. By Proposition~\ref{prop:consol-preserves-cuts-when-tame}, we have that $b$ is also $\per^\circ$-elliptic. In particular, we must necessarily have that $b$ separates $H_1$ from $H_2$. 
\end{proof}

One major benefit of working with separable peripheral systems that we gain access to the following dichotomy for pairs of ends.

\begin{proposition}\label{prop:end-separating-dichotomy}
    Let $\Gamma$ be a connected, locally finite graph and $\per$ a separable, tame peripheral system. Then for all distinct  $\omega_1, \omega_2 \in \Omega(\Gamma)$, exactly one of the following holds:
    \begin{enumerate}
        \item There exists $b \in \br_\per(\Gamma)$ separating $\omega_1$ and $\omega_2$, or

        \item There exists $H \in \per$ such that $\omega_1, \omega_2 \in \Lambda(H)$.
    \end{enumerate}
\end{proposition}

\begin{proof}
    Let $H \in \per$, and $\omega \in \Omega(\Gamma) \setminus \Lambda(H)$. Of course, if no such $\omega$ exists then there is nothing to prove. If there exists $H' \in \per$ such that $\omega \in \Lambda(H')$, then since $\per$ is separable we are done. Thus, we may assume without loss of generality that no such $H'$ exists. 

    Since $\Lambda(H)$ is closed in $\Omega(\Gamma)$, there exists a clopen set $W  \ni \omega$ such that $\Lambda(H) \cap W = \emptyset$. Write $W' = \Omega(\Gamma) \setminus W$. Choose $b \in \br(\Gamma)$ such that $\Lambda(b) = W'$ and $\Lambda(b^\ast) = W$. If $b$ is $\per$-elliptic then we are done, else there exist peripherals in $\per$ which have infinite intersection with both $b$ and $b^\ast$. Since $\per$ is tame, this collection is finite, say $H_1, \ldots, H_k \in \per$. For each $H_j$, by separability we have that there exists $c_j \in \br_\per(\Gamma)$ such that $c_j \cap H_j$ is finite and $c_j ^\ast \cap H$ is finite, i.e. $c_j$ separates $H_j$ from $H$. 
    Consider now the cut
    $$
    a := b \cap \bigcap_{j=1}^k c_j. 
    $$
    We first observe that $a$ is $\per$-elliptic. Indeed, suppose some $H'\in \per$ had infinite intersection with both $a$ and $a^\ast$. Then $H'$ has infinite intersection with $b$ and with all of the $c_j$. Since the $c_j$ are $\per$-elliptic, we have that $H' \cap c_j^\ast$ is finite. Thus, $H' \cap b^\ast$ is infinite. But then $H' = H_j$ for some $j$. This contradicts the fact that $c_j \cap H_j$ is finite. Thus $a$ is $\per$-elliptic. Also, observe that $a$ separates $\omega$ from $\Lambda(H)$. Indeed, we have that $\Lambda(H) \subset \Lambda(a)$ since $\Lambda(H) \subset \Lambda(b) \cap \bigcap_j \Lambda(c_j)$, and $\omega \in \Lambda(b^\ast) \subset \Lambda(a^\ast)$. This proves the proposition in the case where at least one of the $\omega_i$ is contained in the limit set of some peripheral. 

    Now, fix distinct ends $\omega_1$, $\omega_2$ such that neither $\omega_i$ is contained in the limit set of any peripheral. We will proceed very similarly, applying the previous case to aid us. 
    Fix $b \in \br(\Gamma)$ which separates $\omega_1$ and $\omega_2$, say $\omega_1 \in \Lambda(b)$, $\omega_2 \in \Lambda(b^\ast)$. If $b$ is $\per$-elliptic then we are done. Otherwise there are finitely many $H_1, \ldots, H_k \in \per$ which have infinite intersection with both $b$ and $b^\ast$. By assumption, neither $\omega_i$ is contained in the limit set of any of the $H_j$. Applying the previous case, for each $j$ there exists $c_j \in \br_\per(\Gamma)$ such that $\omega_1 \in \Lambda(c_j)$ and $\Lambda(H_j) \subset \Lambda(c_j^\ast)$. As before, consider the cut $a := b \cap \bigcap_{j} c_j$. We can verify in exactly the same way that $a$ is $\per$-elliptic, and separates $\omega_1$ from $\omega_2$. The proposition now follows. 
\end{proof}

Everything above allows us to prove the following helpful result. 

\begin{theorem}\label{thm:wlog-reduced}
    Let $\Gamma$ be a connected, locally finite, quasi-transitive $G$-graph. Let $\per$ be a thin, $G$-invariant peripheral system. Then there exists a tame, minimised, separable, $G$-invariant peripheral system $\per'$ such that $\br_{\per'}(\Gamma) = \br_\per(\Gamma)$. 
\end{theorem}

\begin{proof}
    Let $\per' = \per^{\bullet\circ}$. That is, $\per'$ is the consolidation of the minimisation of $\per$. Since $\per$ is tame, we have by Theorem~\ref{lem:thin-min-implies-tame} that $\per^\bullet$ is tame, and thus by Proposition~\ref{prop:tameness-preserved} we have that $\per'$ is also tame. We trivially have that $\per^\bullet$ is minimised, from which it immediately follows that $\per'$ is minimised. By Proposition~\ref{prop:actually-seperable}, since $\per^\bullet$ is tame we have that $\per'$ is separable. 
    Finally, we have by Propositions~\ref{prop:min-preserves-cuts} and \ref{prop:consol-preserves-cuts-when-tame}, that 
    $$
    \br_\per(\Gamma) = \br_{\per^\bullet}(\Gamma) = \br_{\per'}(\Gamma).
    $$
    The second equality uses the fact that $\per^\bullet$ is tame. The theorem follows.
\end{proof}

\subsection{Coning-off a peripheral system}\label{sec:cone-off}

We now give an alternative interpretation of relative Boolean ring. Given a connected graph $\Gamma$ and a peripheral system $\per$, form a new graph $\widehat \Gamma_\per$ as follows:
\begin{enumerate}
    \item For every $H \in \per$ add a new vertex $v_H$ to $\Gamma$, called the \textit{cone vertex at $H$.}

    \item Connect every $u \in H$ to $v_H$ by an edge. 
\end{enumerate}
The graph $\widehat \Gamma_\per$ is called the \emph{cone of $\Gamma$ over $\per$}. For the present paper, the utility of this construction lies in the following. 

\begin{proposition}\label{prop:ring-isomorphism}
    Let $\Gamma$ be a connected graph and $\per$ a peripheral system. Then the inclusion $\Gamma \into \widehat \Gamma_\per$ induces a ring homomorphism
    $$
    F : \br (\widehat \Gamma_{\per}) \to \br_{\per} (\Gamma). 
    $$
    We have that $F$ is surjective if and only if $\per$ is tame, and $F$ is injective if and only if every $H \in \per$ is infinite. If $\Gamma$ is a $G$-graph and $\per$ is $G$-invariant then $F$ is $G$-equivariant. 
\end{proposition}

\begin{proof}
    Let $F : \br (\widehat \Gamma_{\per}) \to \br (\Gamma)$ be given by restriction. That is, 
    $F(b) = b \cap V(\Gamma) $.
    It is routine to check that this is a ($G$-equivariant) homomorphism of Boolean rings. 
    
    We claim that the image of $F$ is contained in $\br_{\per} (\Gamma)$. Firstly, suppose for the sake of a contradiction that there exists some $b \in \br (\widehat \Gamma_{\per})$ and some $H \in \per$ such that both $H \cap F(b)$ and $H \cap F(b^\ast)$ are infinite. Now, either $v_H \in b$ or $v_H \not\in b$. In either case we see that $\delta b$ is infinite, a contradiction. The claim follows. 

    Suppose now that $\per$ is tame. We claim that $F$ is surjective. Fix $b \in \br_\per \Gamma$. We extend $b$ to an element $b' \in \br (\widehat \Gamma_{\per})$ by including a cone vertex $v_H$ if and only if $b \cap H$ is infinite. It follows from the tameness of $\per$ that $\delta b'$ is finite. Clearly $F(b') = b$, and so $F$ is surjective. Conversely, suppose $\per$ is not tame, so there exists $b \in \br_\per(\Gamma)$ such that both $b \cap H$ and $b^\ast \cap H$ are non-empty for infinitely many $H \in \per$. Let $S\subset \per$ denote the set of elements of $\per$ which intersect both $b$ and $b^\ast$.  Suppose there exists $b' \in F^{-1}(b)$. Then, for all $H \in S$, we have that each $v_H$ is abutted by an edge lying in $\delta b'$. Since no two of the $v_H$ are adjacent, these edges contributed to $\delta b'$ are pairwise distinct. Thus, since $S$ is infinite we conclude that $\delta b'$ must be infinite and so $b' \not\in \br(\widehat \Gamma _\per)$. It follows that such a $b'$ does not exist, and so $F$ is not surjective. 

    Finally, suppose every $H \in \per$ is infinite. We need only show that $F$ has trivial kernel. Here, the 0-element is the empty set. Clearly any $b \in \br(\widehat \Gamma_\per)$ which satisfies $F(b) = \emptyset$ cannot contain any cone vertices lest $b$ must have infinite coboundary. Thus, any element $b$ such that $F(b)$ is empty must itself be empty. Conversely, suppose some $H \in \per$ is finite. Then the singleton $b = \{v_H\}$ has finite coboundary and so is an element of $\br(\widehat \Gamma_\per)$. However, $F(b) = \emptyset$, and so $F$ has non-trivial kernel and is not injective.
\end{proof}

We remark that Proposition~\ref{prop:ring-isomorphism} is incredibly useful in this setting, since it allows us to reduce most questions about $\br_\per$ to questions about $\br$ on a different graph. It is this proposition which motivated the definition of a tame peripheral. As a first application, we learn the following strong fact for free. 

\begin{theorem}\label{thm:nested-gen-set-rel}
    Let $\Gamma$ be a connected $G$-graph and $\per$ a tame, $G$-invariant peripheral system. Then $\br_\per(\Gamma)$ contains a $G$-invariant, nested set which additively generates $\br_\per(\Gamma)$.
\end{theorem}

\begin{proof}
    We know by Theorem~\ref{thm:dicks-dunwoody} and Proposition~\ref{prop:gen-set-upgrade} that $\br(\widehat \Gamma_\per)$ admits a $G$-invariant, nested, additive generating set $E$. Let $F : \br (\widehat \Gamma_{\per}) \to \br_{\per} (\Gamma)$ be the $G$-equivariant homomorphism introduced in Proposition~\ref{prop:ring-isomorphism}. Since $\per$ is tame, $F$ is surjective. Clearly $F$ preserves nesting. We may thus take $F(E)$ as our desired generating set.
\end{proof}

Note that we cannot generally obtain a generating set consisting of tight elements, contrary to the situation for $\br(\Gamma)$. Indeed, it could be the case that $\br_\per(\Gamma)$ contains no tight elements at all. 

\begin{example}
    Let $\Gamma$ denote the 4-regular tree, viewed as the standard Cayley graph of the free group $G = F_2 = \pres{a,b}{-}$. Consider the peripheral system $\per$ consisting of left cosets of the cyclic subgroups $\langle a \rangle$ and $\langle b \rangle$. This is depicted in Figure~\ref{fig:no-tight-cuts}. Every elliptic cut in this graph must have at least two edges in its coboundary, and thus cannot be tight.
\end{example}

\begin{figure}
     \centering
    \tikzset{every picture/.style={line width=0.75pt}} 

\begin{tikzpicture}[x=0.75pt,y=0.75pt,yscale=-1,xscale=1]

\draw  [color={rgb, 255:red, 74; green, 144; blue, 226 }  ,draw opacity=1 ][fill={rgb, 255:red, 232; green, 242; blue, 255 }  ,fill opacity=1 ][dash pattern={on 3.75pt off 1.5pt}] (346.33,185.13) .. controls (348.01,185.13) and (349.38,186.49) .. (349.38,188.18) -- (349.38,197.33) .. controls (349.38,199.01) and (348.01,200.38) .. (346.33,200.38) -- (274.18,200.38) .. controls (272.49,200.38) and (271.13,199.01) .. (271.13,197.33) -- (271.13,188.18) .. controls (271.13,186.49) and (272.49,185.13) .. (274.18,185.13) -- cycle ;
\draw  [color={rgb, 255:red, 74; green, 144; blue, 226 }  ,draw opacity=1 ][fill={rgb, 255:red, 232; green, 242; blue, 255 }  ,fill opacity=1 ][dash pattern={on 3.75pt off 1.5pt}] (346.33,79.13) .. controls (348.01,79.13) and (349.38,80.49) .. (349.38,82.18) -- (349.38,91.32) .. controls (349.38,93.01) and (348.01,94.38) .. (346.33,94.38) -- (274.18,94.38) .. controls (272.49,94.38) and (271.13,93.01) .. (271.13,91.32) -- (271.13,82.18) .. controls (271.13,80.49) and (272.49,79.13) .. (274.18,79.13) -- cycle ;
\draw  [color={rgb, 255:red, 74; green, 144; blue, 226 }  ,draw opacity=1 ][fill={rgb, 255:red, 232; green, 242; blue, 255 }  ,fill opacity=1 ][dash pattern={on 3.75pt off 1.5pt}] (405.2,132.13) .. controls (406.88,132.13) and (408.25,133.49) .. (408.25,135.18) -- (408.25,144.32) .. controls (408.25,146.01) and (406.88,147.38) .. (405.2,147.38) -- (215.05,147.38) .. controls (213.37,147.38) and (212,146.01) .. (212,144.32) -- (212,135.18) .. controls (212,133.49) and (213.37,132.13) .. (215.05,132.13) -- cycle ;
\draw  [color={rgb, 255:red, 208; green, 2; blue, 27 }  ,draw opacity=1 ][fill={rgb, 255:red, 255; green, 237; blue, 237 }  ,fill opacity=1 ][dash pattern={on 3.75pt off 1.5pt}] (360.38,102.3) .. controls (360.38,100.62) and (361.74,99.25) .. (363.43,99.25) -- (372.58,99.25) .. controls (374.26,99.25) and (375.63,100.62) .. (375.63,102.3) -- (375.63,177.2) .. controls (375.63,178.88) and (374.26,180.25) .. (372.58,180.25) -- (363.43,180.25) .. controls (361.74,180.25) and (360.38,178.88) .. (360.38,177.2) -- cycle ;
\draw  [color={rgb, 255:red, 208; green, 2; blue, 27 }  ,draw opacity=1 ][fill={rgb, 255:red, 255; green, 237; blue, 237 }  ,fill opacity=1 ][dash pattern={on 3.75pt off 1.5pt}] (244.88,102.3) .. controls (244.88,100.62) and (246.24,99.25) .. (247.93,99.25) -- (257.08,99.25) .. controls (258.76,99.25) and (260.13,100.62) .. (260.13,102.3) -- (260.13,177.2) .. controls (260.13,178.88) and (258.76,180.25) .. (257.08,180.25) -- (247.93,180.25) .. controls (246.24,180.25) and (244.88,178.88) .. (244.88,177.2) -- cycle ;
\draw  [color={rgb, 255:red, 208; green, 2; blue, 27 }  ,draw opacity=1 ][fill={rgb, 255:red, 255; green, 237; blue, 237 }  ,fill opacity=1 ][dash pattern={on 3.75pt off 1.5pt}] (302.63,53.05) .. controls (302.63,51.37) and (303.99,50) .. (305.68,50) -- (314.83,50) .. controls (316.51,50) and (317.88,51.37) .. (317.88,53.05) -- (317.88,223.95) .. controls (317.88,225.63) and (316.51,227) .. (314.83,227) -- (305.68,227) .. controls (303.99,227) and (302.63,225.63) .. (302.63,223.95) -- cycle ;
\draw    (310.25,86.75) -- (310.25,139.75) ;
\draw [shift={(310.25,139.75)}, rotate = 90] [color={rgb, 255:red, 0; green, 0; blue, 0 }  ][fill={rgb, 255:red, 0; green, 0; blue, 0 }  ][line width=0.75]      (0, 0) circle [x radius= 2.01, y radius= 2.01]   ;
\draw [shift={(310.25,86.75)}, rotate = 90] [color={rgb, 255:red, 0; green, 0; blue, 0 }  ][fill={rgb, 255:red, 0; green, 0; blue, 0 }  ][line width=0.75]      (0, 0) circle [x radius= 2.01, y radius= 2.01]   ;
\draw    (310.25,139.75) -- (310.25,192.75) ;
\draw [shift={(310.25,192.75)}, rotate = 90] [color={rgb, 255:red, 0; green, 0; blue, 0 }  ][fill={rgb, 255:red, 0; green, 0; blue, 0 }  ][line width=0.75]      (0, 0) circle [x radius= 2.01, y radius= 2.01]   ;
\draw [shift={(310.25,139.75)}, rotate = 90] [color={rgb, 255:red, 0; green, 0; blue, 0 }  ][fill={rgb, 255:red, 0; green, 0; blue, 0 }  ][line width=0.75]      (0, 0) circle [x radius= 2.01, y radius= 2.01]   ;
\draw    (368,139.75) -- (310.25,139.75) ;
\draw [shift={(310.25,139.75)}, rotate = 180] [color={rgb, 255:red, 0; green, 0; blue, 0 }  ][fill={rgb, 255:red, 0; green, 0; blue, 0 }  ][line width=0.75]      (0, 0) circle [x radius= 2.01, y radius= 2.01]   ;
\draw [shift={(368,139.75)}, rotate = 180] [color={rgb, 255:red, 0; green, 0; blue, 0 }  ][fill={rgb, 255:red, 0; green, 0; blue, 0 }  ][line width=0.75]      (0, 0) circle [x radius= 2.01, y radius= 2.01]   ;
\draw    (310.25,139.75) -- (252.5,139.75) ;
\draw [shift={(252.5,139.75)}, rotate = 180] [color={rgb, 255:red, 0; green, 0; blue, 0 }  ][fill={rgb, 255:red, 0; green, 0; blue, 0 }  ][line width=0.75]      (0, 0) circle [x radius= 2.01, y radius= 2.01]   ;
\draw [shift={(310.25,139.75)}, rotate = 180] [color={rgb, 255:red, 0; green, 0; blue, 0 }  ][fill={rgb, 255:red, 0; green, 0; blue, 0 }  ][line width=0.75]      (0, 0) circle [x radius= 2.01, y radius= 2.01]   ;
\draw    (252.5,139.75) -- (218.5,139.75) ;
\draw    (402,139.75) -- (368,139.75) ;
\draw    (310.25,86.75) -- (310.25,57.75) ;
\draw    (368,139.75) -- (368,110.75) ;
\draw    (368,168.75) -- (368,139.75) ;
\draw    (310.25,221.75) -- (310.25,192.75) ;
\draw    (252.5,168.75) -- (252.5,139.75) ;
\draw    (252.5,139.75) -- (252.5,110.75) ;
\draw    (310.25,86.75) -- (276.25,86.75) ;
\draw    (344.25,86.75) -- (310.25,86.75) ;
\draw    (344.25,192.75) -- (310.25,192.75) ;
\draw    (310.25,192.75) -- (276.25,192.75) ;
\draw  [draw opacity=0][fill={rgb, 255:red, 255; green, 255; blue, 255 }  ,fill opacity=1 ] (207.5,130.5) -- (214.75,130.5) -- (214.75,150.25) -- (207.5,150.25) -- cycle ;
\draw  [draw opacity=0][fill={rgb, 255:red, 255; green, 255; blue, 255 }  ,fill opacity=1 ] (405.25,130) -- (412.5,130) -- (412.5,149.75) -- (405.25,149.75) -- cycle ;
\draw  [draw opacity=0][fill={rgb, 255:red, 255; green, 255; blue, 255 }  ,fill opacity=1 ] (346.75,182.75) -- (354,182.75) -- (354,202.5) -- (346.75,202.5) -- cycle ;
\draw  [draw opacity=0][fill={rgb, 255:red, 255; green, 255; blue, 255 }  ,fill opacity=1 ] (265.75,183.25) -- (273,183.25) -- (273,203) -- (265.75,203) -- cycle ;
\draw  [draw opacity=0][fill={rgb, 255:red, 255; green, 255; blue, 255 }  ,fill opacity=1 ] (265.75,77.5) -- (273,77.5) -- (273,97.25) -- (265.75,97.25) -- cycle ;
\draw  [draw opacity=0][fill={rgb, 255:red, 255; green, 255; blue, 255 }  ,fill opacity=1 ] (347.25,77) -- (354.5,77) -- (354.5,96.75) -- (347.25,96.75) -- cycle ;
\draw  [draw opacity=0][fill={rgb, 255:red, 255; green, 255; blue, 255 }  ,fill opacity=1 ] (359,96.25) -- (377.25,96.25) -- (377.25,104) -- (359,104) -- cycle ;
\draw  [draw opacity=0][fill={rgb, 255:red, 255; green, 255; blue, 255 }  ,fill opacity=1 ] (358.5,175.5) -- (379,175.5) -- (379,183.25) -- (358.5,183.25) -- cycle ;
\draw  [draw opacity=0][fill={rgb, 255:red, 255; green, 255; blue, 255 }  ,fill opacity=1 ] (242,175.25) -- (262.5,175.25) -- (262.5,183) -- (242,183) -- cycle ;
\draw  [draw opacity=0][fill={rgb, 255:red, 255; green, 255; blue, 255 }  ,fill opacity=1 ] (300.5,224.5) -- (321,224.5) -- (321,232.25) -- (300.5,232.25) -- cycle ;
\draw  [draw opacity=0][fill={rgb, 255:red, 255; green, 255; blue, 255 }  ,fill opacity=1 ] (300.25,46.5) -- (320.75,46.5) -- (320.75,54.25) -- (300.25,54.25) -- cycle ;
\draw  [draw opacity=0][fill={rgb, 255:red, 255; green, 255; blue, 255 }  ,fill opacity=1 ] (242,96.75) -- (262.5,96.75) -- (262.5,104.5) -- (242,104.5) -- cycle ;

\end{tikzpicture}
    \caption{Peripheral system on the 4-regular tree with no tight elliptic cuts.}
    \label{fig:no-tight-cuts}
\end{figure}
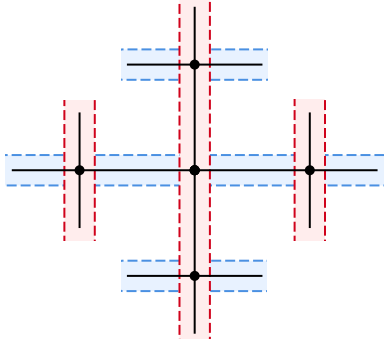

We conclude this subsection by noting the following corollary of Theorem~\ref{thm:nested-gen-set-rel}, when combined with Proposition~\ref{prop:gen-set-upgrade}. 

\begin{corollary}\label{cor:gen-set-upgrade-rel-cuts}
    Let $\Gamma$ be a connected $G$-graph and $\per$ a tame, $G$-invariant peripheral system. Suppose $\br_\per(\Gamma)$ is $G$-finitely generated as a ring. Then $\br_\per(\Gamma)$ is a finitely generated $G$-module. 
\end{corollary}

\begin{proof}
    Let $S= \{s_1, \ldots, s_k\} \subset \br_\per(\Gamma)$ be a finite set such that $G \cdot S$ generates $\br_\per(\Gamma)$ as a ring. Let $ E \subset \br _\per(\Gamma)$ be a nested, $G$-invariant generating set, given by Theorem~\ref{thm:nested-gen-set-rel}.  For each $s_i \in S$, we may express $s_i$ as a sum of elements in $E$, say $s_i = \sum_ie_{ij}$. But then $S' = G \cdot \{e_{ij}\}_{ij}$ is a $G$-invariant, $G$-finite, \textbf{nested} subset of $\br_\per(\Gamma)$ which generates $\br_\per(\Gamma)$ as a ring. We may assume without harm that $S'$ is closed under taking complements, as adding these complements doesn't affect the fact that $S'$ is $G$-invariant, $G$-finite, and nested. By Proposition~\ref{prop:gen-set-upgrade}, $S'$ generates $\br_\per(\Gamma)$ as an abelian group, and so since $S'$ is $G$-finite, we have that $\br_\per(\Gamma)$ is a finitely generated $G$-module. 
\end{proof}

\subsection{Relative accessibility}

We now define what it means for a graph to be accessible relative to a peripheral system.

\begin{definition}[Relative accessibility]
    Let $\Gamma$ be a connected, locally finite graph, and $\per$ a peripheral system.
    We then say that $\Gamma$ is \emph{accessible relative to $\per$} if there exists $k \in \N$ such that for every pair of distinct ends $\omega_1$, $\omega_2$ of $\Gamma$, if there exists an elliptic cut $b$ separating $\omega_1$ and $\omega_2$, then there exists one satisfying $|\delta b| \leq k$. 
\end{definition}

The goal of this section is to characterise relative accessibility in quasi-transitive graphs in terms of the generation properties of the subring $\br_\per$. 
To this end, we will prove the following theorem, which is one of the main results of this paper. This should be compared with {\cite[Prop.~6.2]{thomassen1993vertex}}.

\fgchar*

We will prove this in a few stages, as its surprisingly subtle. Throughout the remainder of this section, fix $\Gamma$ be a connected, locally finite, quasi-transitive $G$-graph, and $\per$ a thin, $G$-invariant peripheral system. 
By Theorem~\ref{thm:wlog-reduced}, we may assume without loss of generality that $\per$ is tame, minimised, and separable. The rough plan is then as follows. First, we may bipartition $\per$ into two parts $\per = \mathcal I \sqcup \mathcal T$, where
$$
\mathcal I = \{H \in \per : |\Lambda(H)| = \infty\}, \ \  \mathcal T = \{H \in \per : |\Lambda(H)| = 2\}.
$$
If $\mathcal T = \emptyset$, then the proof becomes substantially shorter (see Lemmas~\ref{claim:coboundary-control}, \ref{claim:inf-ended-case}). The task is to somehow reduce to this case, and show that we may safely delete $\mathcal T$ without affecting relative accessibility. Note that $\br_{\mathcal I}(\Gamma) \neq \br _\per(\Gamma)$ in general, so this reduction is a bit subtle. 

By Lemma~\ref{lem:unbounded-coarse-comps} and Remark~\ref{rem:replace-with-neighbourhoods}, we may assume without loss of generality that for every $H' \in \mathcal I$, we have that every connected component $\Delta$ of $\Gamma[H']$ satisfies $|\Lambda(\Delta)| = \infty$. 

We first observe the easy direction.

    \begin{lemma}\label{claim:easy-case}
        If $\br_\per(\Gamma)$ is a finitely generated $G$-module then $\Gamma$ is accessible relative to $\per$.
    \end{lemma}

    \begin{proof}
        Let $S$ be a $G$-invariant, $G$-finite generating set of  $\br_\per(\Gamma)$ as an abelian group. Let $\omega_1$, $\omega_2$ be two ends and $b \in \br_\per(\Gamma)$ be an $\per$-elliptic cut which separates them. Then, by Proposition~\ref{prop:separating-generators}, there exists $b' \in S$ which separates $\omega_1$ from $\omega_2$. Since $S$ is $G$-finite, there is a uniform bound on $|\delta b'|$. The claim follows. 
    \end{proof}

We now proceed with the more difficult direction. We will begin by first only considering $\mathcal I$, ignoring the two-ended peripherals contributed by $\mathcal T$. Later, we will see why this is sufficient. 
    Let 
    $$
    F : \br (\widehat \Gamma_{\mathcal I}) \to \br_{\mathcal I} (\Gamma)
    $$
    denote the isomorphism given by Proposition~\ref{prop:ring-isomorphism}. To ease notation, we will use a caret to denote the inverse $F^{-1}$; i.e. $F^{-1}(b) =: \widehat b$. Our first claim establishes some quantitative control on how this map effects coboundaries.

 \begin{lemma}\label{claim:coboundary-control}
        There exists an increasing function $\rho : \N \to \N$ such that for all $b \in \br_{\mathcal I}(\Gamma)$, we have
        $$
         |\delta \widehat b| \leq \rho(|\delta b| ).
        $$
    \end{lemma}

    \begin{proof}
        We will prove this using an isoperimetric argument.
        If the claim were false, there would exist a sequence $(b_n)$ in $\br_{\mathcal I}(\Gamma)$, with $|\delta b_n|$ uniformly bounded, such that $|\delta \widehat b_n| \to \infty$. Without loss of generality, we may assume that $|\delta b_n| = |\delta b_m|$ for all $m, n > 0$.

        Note that, given $b \in \br_{\mathcal I}(\Gamma)$, we may compute $|\delta \widehat b|$ via the formula
        $$
        |\delta \widehat b| = |\delta b| + \sum_H |b \cap H| + \sum_{H'} |b^\ast \cap H'|,
        $$
        where the first sum is taken over those $H \in \mathcal I$ such that $b \cap H$ is finite, and the second is taken over those $H' \in \mathcal I$ such that $b^\ast \cap H'$ is finite. 

        Back to our sequence $(b_n)$, we may assume without loss of generality that, say, the first sum $\sum_H |b_n \cap H| \to \infty$ as $n \to \infty$. Since every $H \in \mathcal I$ has big components, we have that for each $n > 0$ there exists a $H_n \in \mathcal I$, and a component $K_n \subset H_n$, such that  $b_n \cap K_n|$ is finite for all $n$, and $|b_n \cap K_n| \to \infty$. Passing to another subsequence, we may assume without loss of generality that all $K_n$ are $G$-orbit equivalent. In particular, by translating we may assume that $K_n = K_m$ for all $m,n > 0$. Let us write $K_n = K$ to simplify notation. Note that $\Delta := \Gamma[K]$ itself is a connected, locally finite, quasi-transitive graph. However, $V(\Delta)$ contains a sequence of arbitrarily large finite subsets $A_n = b_n \cap K$ such that $\delta A_n$ is uniformly bounded. This implies that $\Delta$ is two-ended. In particular, $|\Lambda(K)| \leq 2$, but this contradicts the fact that $|\Lambda(K)| =\infty$. The claim follows.
    \end{proof}

    We will use Lemma~\ref{claim:coboundary-control} to prove the following.

    \begin{lemma}\label{claim:inf-ended-case}
        Suppose $\Gamma$ is accessible relative to $\mathcal I$. Then $\br_{\mathcal I}(\Gamma)$ is a finitely generated $G$-module.
    \end{lemma}

    \begin{proof}
        By relative accessibility, quasi-transitivity, and Proposition~\ref{prop:g-finite-criterion}, let $S \subset \br_{\mathcal I}(\Gamma)$ be a subset with the following properties:
        \begin{enumerate}
            \item $|\delta b|$ is uniformly bounded across all $b \in S$.

            \item $S$ generates $ \br_{\mathcal I}(\Gamma)$ as a ring.
        \end{enumerate}
        Consider the lift $\widehat S := F^{-1}(S)$. We have that $\widehat S$ generates $\br(\widehat\Gamma_{\mathcal I})$ as a ring. By Lemma~\ref{claim:coboundary-control}, $\widehat S \subset  \br_n(\widehat\Gamma_{\mathcal I})$ for some finite $n > 0$, and so we deduce that
        $$
        \br_n(\widehat\Gamma_{\mathcal I}) =  \br(\widehat\Gamma_{\mathcal I}).
        $$
        Since $\mathcal I$ is thin and the action of $G$ on $\Gamma$ is quasi-transitive, we have that $G$ acts on $E(\widehat \Gamma_{\mathcal I})$ with finitely many orbits. By Corollary~\ref{cor:dicks-dunwoody-gfinite}, it follows that  $\br_n(\widehat\Gamma_{\mathcal I}) =  \br(\widehat\Gamma_{\mathcal I})$ contains a nested, $G$-invariant, $G$-finite additive generating set $S'$. Taking $S'' = F(S')$, we conclude that $\br_{\mathcal I}(\Gamma)$ is a finitely generated $G$-module. 
    \end{proof}

    This essentially completes the proof in the case where $\mathcal T = \emptyset$. However, in the presence of two-ended peripherals there is still some work to be done.
    We now have the following three claims.

    \begin{lemma}\label{claim:acc-H-implies-acc-I}
        If $\Gamma$ is accessible relative to $\per$, then $\Gamma$ is accessible relative to $\mathcal I$.
    \end{lemma}

    \begin{proof}
        Let $S \subset \br_\per(\Gamma)$ satisfy the following:
        \begin{enumerate}
            \item There exists $k > 0$ such that $|\delta b| < k$ for all $b \in S$.

            \item For all distinct $\omega_1, \omega_2 \in \Omega(\Gamma)$, either there exists $H \in \per$ such that $\omega_1, \omega_2 \in \Lambda(H)$, or there is some $b \in S$ which separates $\omega_1$ and $\omega_2$. 
        \end{enumerate}
        The existence of such an $S$ is guaranteed by the fact that $\Gamma$ is accessible relative to $\per$, and Proposition~\ref{prop:end-separating-dichotomy}. Let $H_1, \ldots, H_l \in \mathcal T$ be orbit representatives, so $|\Lambda(H_i)|=2$ for each $1 \leq i \leq l$. Again, by Proposition~\ref{prop:end-separating-dichotomy}, for each such $i$ there exists an $\mathcal I$-elliptic cut $b_i \in \br_{\mathcal I}(\Gamma)$ such that $b_i$ separates the two ends in $\Lambda(H_i)$. 
        We now augment $S$ via
        $$
        S' := S \cup \{g b_i : g \in G, 1 \leq i \leq l\}. 
        $$
        We have that any two ends in $\Gamma$ are either contained in the limit set of some $K \in \mathcal I$, or are separated by some $b \in S'$. Moreover, the size of the coboundaries of elements of $S'$ is certainly uniformly bounded. It follows that $\Gamma$ is accessible relative to $\mathcal I$. 
    \end{proof}

The next two lemmas are relatively technical.

    \begin{lemma}\label{claim:finding-the-special-cut}
        Let $H \in \mathcal T$. Let $U \supset \Lambda(H)$ be an open neighbourhood of the limit set of $H$ in $\Omega(\Gamma)$. Then there exists $b \in \br_\per(\Gamma)$ such that for all $\omega \in \Omega(\Gamma) \setminus U$, we have that $b$ separates $\omega$ from $\Lambda(H)$. 
    \end{lemma}

    \begin{proof}
        For each $p \in \Omega(\Gamma)\setminus U$, we have by Proposition~\ref{prop:end-separating-dichotomy} that there exists $b_p \in \br_\per(\Gamma)$ such that $p \in \Lambda(b_p) =: W_p$ and $\Lambda(H) \subset \Lambda(b^\ast_p)$. Note that $W_p$ is clopen, and so
        $$
        U \cup \bigcup_p W_p
        $$
        is an open cover of $\Omega(\Gamma)$. Since $\Omega(\Gamma)$ is compact, there exists a finite subcover, say
        $$
        U \cup \bigcup_{i=1}^n W_i,
        $$
        where $W_i := W_{p_i}$ for some $p_i \in \Omega(\Gamma) \setminus U$. Write $b_i := b_{p_i}$, and let
        $$
        b = \bigg(\bigcap_i b_i^\ast\bigg)^\ast = \bigcup_i b_i.
        $$
        Note that $b \in \br_\per(\Gamma)$  since $\br_\per(\Gamma)$ is a subring of $\br(\Gamma)$ by Proposition~\ref{prop:subring}. It is immediate that $b$ satisfies our requirements. 
    \end{proof}

    \begin{lemma}\label{claim:technical-boundary-claim}
        Let $H_1, \ldots, H_k \in \mathcal T$, with $\{\lambda_{j,1}, \lambda_{j,2}\} = \Lambda(H_j)$. Let $W_1 \sqcup W_2 = \Omega(\Gamma)$ be a bipartition into clopen sets, such that $\lambda_{j,i} \in W_i$ for all $j$, $i$. Then there exists pairwise disjoint, open neighbourhoods $U_{j,i} \ni \lambda_{j,i}$ in $\Omega(\Gamma)$, with $U_{j,i} \subset W_i$, such that the following holds:
            
            For all $i \in \{1,2\}$, $j \in \{1,\ldots, k\}$, $p \in U_{j,i} \setminus \Lambda(H_j)$, there exists $g \in G_{H_j}$ such that:
            \begin{enumerate}
                \item\label{itm:dynamics-property-1} $gp \in W_i \setminus \bigcup _\ell U_{\ell, i}$, and

                \item\label{itm:dynamics-property-2} $gU_{j,i+1} \cap W_i= \emptyset$,
            \end{enumerate}
            where addition within the indices is taken modulo 2.
    \end{lemma}

    \begin{proof}
        We will focus on the case $i = 1$, then deal with $i = 2$ via a symmetry argument at the end. 
        First, note that the $\lambda_{j,1}$ are pairwise distinct. Thus, we may choose for each of them an open neighbourhood 
        $$
        \lambda_{j,1} \in Z_j \subset W_1,
        $$
        such that the $Z_j$ are pairwise disjoint. Since $|\Lambda(H_j)| = 2$, by Proposition~\ref{prop:independent-lox} we have that for each $j$ there exists a loxodromic element $t_j \in G_{H_j}$ such that $\Lambda(t_j) = \Lambda(H_j)$. In particular, we have that $t_j$ fixes each of the $\lambda_{j,i}$. 
        Given $j \in \{1, \ldots, k\}$, choose $U_{j,1} \ni \lambda_{j,1}$ such that
        $$
        U_{j,1} \subset Z_{j} \cap t_jZ_{j}.
        $$
        Note that such a $U_{j,1}$ exists since $Z_j$ is open and $\lambda_{j,1} \in Z_{j} \cap t_jZ_{j}$. 
        In particular, this implies that $t_j^{-1}U_{j,1} \subset Z_j$. We may choose such a $U_{j,1}$ since $t_j$ fixes $\Lambda(H_j)$ pointwise. Note that with these choices of $U_{j,1}$, we have already satisfied (\ref{itm:dynamics-property-1}) in the $i=1$ case, by taking $g$ to be some negative power of $t_j$. Moreover, if we replace $U_{j,1}$ with some sub-neighbourhood $ \lambda_{j,1} \in U'_{j,1} \subset U_{j,1}$, then this property is preserved. 

        By Proposition~\ref{prop:lox-action}, we have that there exists $M_j > 0$ such that for all $m > M_j$, we have $t_j^{-m}(W_2) \subset W_2$. 
        Now, for each $m \in \{1, \ldots, M_j\}$, there exists an open neighbourhood $Y_j^{(m)} \ni \lambda_{j,2}$ such that $t_j^{-m}(Y_j^{(m)}) \subset W_2$. Let 
        $$
        U_{j,2} = \bigcap_{m=1}^{M_j} Y_j^{(m)}.
        $$
        It is clear now that $t_j^{-m}(U_{j,2}) \subset W_2$ for all $m > 0$. This verifies (\ref{itm:dynamics-property-2}) in the $i=1$ case. Note again that shrinking the $U_{j,2}$ preserves this property. 

        To deal with the $i=2$ case, we apply the same argument with the roles reserved, and obtain $U'_{j,i} \ni \lambda_{j,i}$. We now replace $U_{j,i} \mapsto U_{j,i} \cap U_{j,i}'$. It is immediate that these choices of $U_{j,i}$ satisfy the claim.
    \end{proof}

    \begin{figure}
         \centering
        \tikzset{every picture/.style={line width=0.75pt}} 

\begin{tikzpicture}[x=0.75pt,y=0.75pt,yscale=-1,xscale=1]

\draw  [color={rgb, 255:red, 155; green, 155; blue, 155 }  ,draw opacity=1 ][fill={rgb, 255:red, 231; green, 231; blue, 231 }  ,fill opacity=1 ] (289.08,151.46) .. controls (289.08,145.04) and (294.29,139.83) .. (300.71,139.83) .. controls (307.13,139.83) and (312.33,145.04) .. (312.33,151.46) .. controls (312.33,157.88) and (307.13,163.08) .. (300.71,163.08) .. controls (294.29,163.08) and (289.08,157.88) .. (289.08,151.46) -- cycle ;
\draw  [color={rgb, 255:red, 208; green, 2; blue, 27 }  ,draw opacity=1 ][fill={rgb, 255:red, 250; green, 219; blue, 219 }  ,fill opacity=1 ] (283.08,108.33) .. controls (283.08,96.04) and (293.04,86.08) .. (305.33,86.08) .. controls (317.62,86.08) and (327.58,96.04) .. (327.58,108.33) .. controls (327.58,120.62) and (317.62,130.58) .. (305.33,130.58) .. controls (293.04,130.58) and (283.08,120.62) .. (283.08,108.33) -- cycle ;
\draw  [color={rgb, 255:red, 144; green, 19; blue, 254 }  ,draw opacity=1 ][fill={rgb, 255:red, 198; green, 143; blue, 244 }  ,fill opacity=0.4 ] (298.33,94) .. controls (318.33,84) and (327.58,90.25) .. (325.33,103.75) .. controls (323.08,117.25) and (306.58,113.75) .. (310.08,128.25) .. controls (313.58,142.75) and (294.33,151.75) .. (294.33,132.25) .. controls (294.33,112.75) and (278.33,104) .. (298.33,94) -- cycle ;
\draw  [color={rgb, 255:red, 155; green, 155; blue, 155 }  ,draw opacity=1 ][fill={rgb, 255:red, 231; green, 231; blue, 231 }  ,fill opacity=1 ] (310.08,64.96) .. controls (310.08,55.91) and (317.41,48.58) .. (326.46,48.58) .. controls (335.5,48.58) and (342.83,55.91) .. (342.83,64.96) .. controls (342.83,74) and (335.5,81.33) .. (326.46,81.33) .. controls (317.41,81.33) and (310.08,74) .. (310.08,64.96) -- cycle ;
\draw  [color={rgb, 255:red, 155; green, 155; blue, 155 }  ,draw opacity=1 ][fill={rgb, 255:red, 231; green, 231; blue, 231 }  ,fill opacity=1 ] (406.38,92.42) .. controls (406.38,85.08) and (412.33,79.13) .. (419.67,79.13) .. controls (427.01,79.13) and (432.96,85.08) .. (432.96,92.42) .. controls (432.96,99.76) and (427.01,105.71) .. (419.67,105.71) .. controls (412.33,105.71) and (406.38,99.76) .. (406.38,92.42) -- cycle ;
\draw  [color={rgb, 255:red, 155; green, 155; blue, 155 }  ,draw opacity=1 ][fill={rgb, 255:red, 231; green, 231; blue, 231 }  ,fill opacity=1 ] (398.72,160.54) .. controls (398.72,141.59) and (414.09,126.22) .. (433.04,126.22) .. controls (452,126.22) and (467.36,141.59) .. (467.36,160.54) .. controls (467.36,179.5) and (452,194.86) .. (433.04,194.86) .. controls (414.09,194.86) and (398.72,179.5) .. (398.72,160.54) -- cycle ;
\draw  [color={rgb, 255:red, 155; green, 155; blue, 155 }  ,draw opacity=1 ][fill={rgb, 255:red, 231; green, 231; blue, 231 }  ,fill opacity=1 ] (359.22,204.04) .. controls (359.22,196.13) and (365.63,189.72) .. (373.54,189.72) .. controls (381.45,189.72) and (387.86,196.13) .. (387.86,204.04) .. controls (387.86,211.95) and (381.45,218.36) .. (373.54,218.36) .. controls (365.63,218.36) and (359.22,211.95) .. (359.22,204.04) -- cycle ;
\draw [color={rgb, 255:red, 74; green, 74; blue, 74 }  ,draw opacity=1 ]   (305.33,108.33) .. controls (357.58,113.75) and (389.58,129.75) .. (433.04,160.54) ;
\draw [color={rgb, 255:red, 74; green, 74; blue, 74 }  ,draw opacity=1 ]   (326.46,64.96) .. controls (349.58,90.25) and (372.08,107.75) .. (419.67,92.42) ;
\draw [color={rgb, 255:red, 74; green, 74; blue, 74 }  ,draw opacity=1 ]   (300.71,151.46) .. controls (340.58,137.75) and (375.58,151.75) .. (373.54,204.04) ;
\draw [color={rgb, 255:red, 128; green, 128; blue, 128 }  ,draw opacity=1 ] [dash pattern={on 3.75pt off 1.5pt}]  (298.08,206.25) .. controls (338.08,176.25) and (369.08,97.75) .. (370.33,58.25) ;
\draw    (373.54,204.04) ;
\draw [shift={(373.54,204.04)}, rotate = 0] [color={rgb, 255:red, 0; green, 0; blue, 0 }  ][fill={rgb, 255:red, 0; green, 0; blue, 0 }  ][line width=0.75]      (0, 0) circle [x radius= 3.35, y radius= 3.35]   ;
\draw    (433.04,160.54) ;
\draw [shift={(433.04,160.54)}, rotate = 0] [color={rgb, 255:red, 0; green, 0; blue, 0 }  ][fill={rgb, 255:red, 0; green, 0; blue, 0 }  ][line width=0.75]      (0, 0) circle [x radius= 3.35, y radius= 3.35]   ;
\draw    (419.67,92.42) ;
\draw [shift={(419.67,92.42)}, rotate = 0] [color={rgb, 255:red, 0; green, 0; blue, 0 }  ][fill={rgb, 255:red, 0; green, 0; blue, 0 }  ][line width=0.75]      (0, 0) circle [x radius= 3.35, y radius= 3.35]   ;
\draw    (305.33,108.33) ;
\draw [shift={(305.33,108.33)}, rotate = 0] [color={rgb, 255:red, 0; green, 0; blue, 0 }  ][fill={rgb, 255:red, 0; green, 0; blue, 0 }  ][line width=0.75]      (0, 0) circle [x radius= 3.35, y radius= 3.35]   ;
\draw    (326.46,64.96) ;
\draw [shift={(326.46,64.96)}, rotate = 0] [color={rgb, 255:red, 0; green, 0; blue, 0 }  ][fill={rgb, 255:red, 0; green, 0; blue, 0 }  ][line width=0.75]      (0, 0) circle [x radius= 3.35, y radius= 3.35]   ;
\draw    (300.71,151.46) ;
\draw [shift={(300.71,151.46)}, rotate = 0] [color={rgb, 255:red, 0; green, 0; blue, 0 }  ][fill={rgb, 255:red, 0; green, 0; blue, 0 }  ][line width=0.75]      (0, 0) circle [x radius= 3.35, y radius= 3.35]   ;
\draw [color={rgb, 255:red, 208; green, 2; blue, 27 }  ,draw opacity=1 ]   (433.58,146.58) ;
\draw [shift={(433.58,146.58)}, rotate = 0] [color={rgb, 255:red, 208; green, 2; blue, 27 }  ,draw opacity=1 ][fill={rgb, 255:red, 208; green, 2; blue, 27 }  ,fill opacity=1 ][line width=0.75]      (0, 0) circle [x radius= 3.35, y radius= 3.35]   ;
\draw [color={rgb, 255:red, 144; green, 19; blue, 254 }  ,draw opacity=1 ]   (375.08,140.83) ;
\draw [shift={(375.08,140.83)}, rotate = 0] [color={rgb, 255:red, 144; green, 19; blue, 254 }  ,draw opacity=1 ][fill={rgb, 255:red, 144; green, 19; blue, 254 }  ,fill opacity=1 ][line width=0.75]      (0, 0) circle [x radius= 3.35, y radius= 3.35]   ;
\draw    (424.08,146.25) .. controls (410.64,146.25) and (398.36,145.79) .. (383.02,141.99) ;
\draw [shift={(381.08,141.5)}, rotate = 14.66] [color={rgb, 255:red, 0; green, 0; blue, 0 }  ][line width=0.75]    (4.37,-1.32) .. controls (2.78,-0.56) and (1.32,-0.12) .. (0,0) .. controls (1.32,0.12) and (2.78,0.56) .. (4.37,1.32)   ;

\draw (308.75,203.65) node [anchor=north west][inner sep=0.75pt]  [font=\scriptsize,color={rgb, 255:red, 128; green, 128; blue, 128 }  ,opacity=1 ]  {$W_{1}$};
\draw (287.25,185.9) node [anchor=north west][inner sep=0.75pt]  [font=\scriptsize,color={rgb, 255:red, 128; green, 128; blue, 128 }  ,opacity=1 ]  {$W_{2}$};
\draw (439,140.9) node [anchor=north west][inner sep=0.75pt]  [font=\footnotesize,color={rgb, 255:red, 208; green, 2; blue, 27 }  ,opacity=1 ]  {$p$};
\draw (367,145.4) node [anchor=north west][inner sep=0.75pt]  [font=\footnotesize,color={rgb, 255:red, 144; green, 19; blue, 254 }  ,opacity=1 ]  {$gp$};
\draw (260.75,102.9) node [anchor=north west][inner sep=0.75pt]  [font=\scriptsize,color={rgb, 255:red, 208; green, 2; blue, 27 }  ,opacity=1 ]  {$U_{j,2}$};
\draw (328.5,93.9) node [anchor=north west][inner sep=0.75pt]  [font=\scriptsize,color={rgb, 255:red, 144; green, 19; blue, 254 }  ,opacity=1 ]  {$gU_{j,2}$};
\draw (421.75,179.15) node [anchor=north west][inner sep=0.75pt]  [font=\scriptsize,color={rgb, 255:red, 128; green, 128; blue, 128 }  ,opacity=1 ]  {$U_{j,1}$};

\end{tikzpicture}
        \caption{Illustration of Lemma~\ref{claim:technical-boundary-claim}.}
        \label{fig:technical-boundary-claim}
    \end{figure}
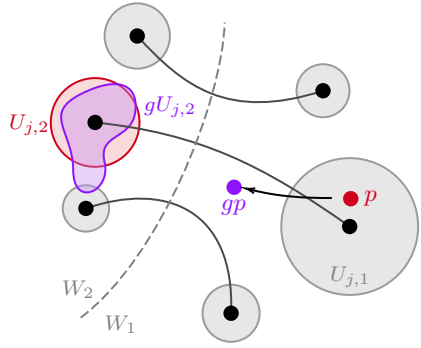

    See Figure~\ref{fig:technical-boundary-claim} for a cartoon of Lemma~\ref{claim:technical-boundary-claim}.
    We will now use Lemmas~\ref{claim:finding-the-special-cut} and \ref{claim:technical-boundary-claim} to prove the following, which shows that we were well within our rights to ignore the two-ended peripherals.

    \begin{lemma}\label{claim:I-fg-implies-H-fg}
        If $\br_{\mathcal I}(\Gamma)$ is a finitely generated $G$-module, then so is $\br_\per(\Gamma)$. 
    \end{lemma}

    \begin{proof}
        Let $S \subset \br_{\mathcal I}(\Gamma)$ be a finite subset so that $G \cdot S$ additively generates $\br_{\mathcal I}(\Gamma)$.  Bipartition $S$ into two sets $S = S_{\mathrm{ell}} \sqcup S _{\mathrm{hyp}}$, where
        $$
        S_{\mathrm{ell}} := \{b \in S : \text{$b$ is $\per$-elliptic}\}, \ \ \ S_{\mathrm{hyp}} := \{b \in S : \text{$b$ is not $\per$-elliptic}\}.
        $$
        Note that if $S_{\mathrm{hyp}} = \emptyset$, then there is nothing to prove, so let's assume this set is non-empty. For each $b \in S_{\mathrm{hyp}}$, let $H_1^{b} \ldots, H_{k_b}^b \in \mathcal T$ denote the (necessarily finite and non-empty) collection of $\mathcal T$-peripherals which have infinite intersection with both $b$ and $b^\ast$. 

        Every $b \in S_{\mathrm{hyp}}$ induces a non-trivial bipartition of $\Omega(\Gamma)$ into clopen sets $W_1^b \sqcup W_2^b$. Let $\Lambda(H_j^b) = \{\lambda_{j,1}^b, \lambda_{j,2}^b\}$, indexed so that $\lambda^b_{j,i} \in W_i^b$. 
        We now apply Claim~\ref{claim:technical-boundary-claim}, and for each $b \in S_{\mathrm{hyp}}$ we obtain pairwise disjoint neighbourhoods $U_{j,i}^b$ of $\lambda_{j,i}^b$, such that $U_{j,i}^b \subset W_i^b$, and such that for all $j \in \{1,\ldots, k_b\}$, and all $p \in U_{i,j}^b$, there exists $g \in G_{H_j^b}$ such that:
            \begin{enumerate}
                \item $gp \in W_1^b \setminus \bigcup _\ell U_{\ell,1}^b$, and

                \item $gU_{j,2}^b \cap W_1^b = \emptyset$.
            \end{enumerate}
        Now, apply Claim~\ref{claim:finding-the-special-cut}, and for each $j \in \{1, \ldots, k_b\}$ find $c^b_j \in \br_\per(\Gamma)$ such that for all $\omega \in \Omega(\Gamma) \setminus (U_{j,1}^b \cup U_{j,2}^b)$, we have that $c_j^b$ separates $\omega$ from $\Lambda(H_j^b)$. These are oriented in such a way that $H_j^b \cap c_j^b$ is finite. 

        Now, consider the set 
        $$
        R = S_{\mathrm{ell}} \cup \Big\{c_j^b : b \in S_{\mathrm{hyp}}, \ j = 1, \ldots, k_b  \Big\} \cup \Bigg\{ b \cap \bigcap_{j=1}^{k_b} c_j^b : b \in S_{\mathrm{hyp}}  \Bigg\}. 
        $$
        Note that $R$ is finite. We have the following two claims.

        \begin{claim}
            We have that $R \subset \br_\per(\Gamma)$.
        \end{claim}

        \begin{proof}
            Indeed, it is sufficient to show that $a = b \cap \bigcap_{j=1}^{k_b} c_j^b $ is $\mathcal T$-elliptic for every $b \in S_{\mathrm{hyp}}$. Suppose some $H \in \mathcal T$ has infinite intersection with both $a$ and $a^\ast$. Then $H$ has infinite intersection with $b$ and with every $c_j^b$. Each $c_j^b$ is $\mathcal H$-elliptic, so since $H \cap a^\ast$ is infinite, we must have that $H \cap b^\ast$ is infinite. But then $H = H_j^b$ for some fixed $j$. But this contradicts the fact that $H \cap c_j^b$ is infinite, which is false by construction. It follows that $R \subset \br_\per(\Gamma)$. 
        \end{proof}

        \begin{claim}
             For all distinct pairs  $\omega_1, \omega_2 \in \Omega(\Gamma)$, either $\omega_1, \omega_2 \in \Lambda(H)$ for some $H \in \per$, or there exists $a \in G \cdot R$, such that $a$ separates $\omega_1$ and $\omega_2$.
        \end{claim}

        \begin{proof}

        Fix $\omega_1$ and $\omega_2$, and assume without loss of generality that they do not lie in the limit set of some common peripheral. After translating, we may assume that $\omega_1$ and $\omega_2$ are separated by some fixed $b \in S$. Assume without loss of generality that $b \in S_{\mathrm{hyp}}$, and that $\omega_i \in W_i^b$. We have a few cases to consider. 

        \textbf{Case 1.} It could be the case that some $c_j^b$ separates $\omega_1$ and $\omega_2$, in which case we are done. 

        \textbf{Case 2.} Suppose that no $c_j^b$ separates $\omega_1$ and $\omega_2$, and that both $\omega_1$ and $\omega_2$ are contained in (the limit set of) the intersection
        $
        \bigcap_{j=1}^{k_b} c_j^b
        $.
        Then it is easy to see that 
        $
        b \cap \bigcap_{j=1}^{k_b} c_j^b 
        $
        separates $\omega_1$ from $\omega_2$. In this case, we are done. 

        \textbf{Case 3.} Suppose that no $c_j^b$ separates $\omega_1$ and $\omega_2$, but the hypotheses of Case 2 do not hold. Note that, by construction, the limit sets of the $(c_{j}^b)^\ast$ are pairwise disjoint, since
        $$
        \Lambda(c^b_j) \cap (U_{j,1}^b \cup U_{j,2}^b) = \emptyset.
        $$
        In particular, in this case we must have that $\omega_i \in U_{j,i}^b$ for some fixed $j \in \{1, \ldots, k_b\}$. By assumption, we do not have that both $\omega_i$ lie in $\Lambda(H_j^b)$, so let's say without loss of generality that $\omega_1$ does not lie in $\Lambda(H_j^b)$. Now, applying Claim~\ref{claim:technical-boundary-claim}, we have by construction that there exists $h \in G_{H_j^b}$ such that $h \omega_1 \in W_1^b \setminus \bigcup _\ell U_{\ell,1}^b$, and $h \omega_2 \in W_2$. Note that $h\omega_1$, $h\omega_2$ are still separated by $b$. By replacing $\omega_1$ and $\omega_2$ with their $h$-translates, we now have that we must be in either Case 1 or Case 2. In particular, $\omega_1$ and $\omega_2$ are separated by some $G$-translate of some element in $R$. 

        These three cases are exhaustive, and so the claim follows.
        \end{proof}
        
        It now follows from Proposition~\ref{prop:g-finite-criterion} and Corollary~\ref{cor:gen-set-upgrade-rel-cuts} that $\br_\per(\Gamma)$ is a finitely generated $G$-module. This concludes the proof of the lemma. 
    \end{proof}

    Theorem~\ref{thm:fg-char} now follows by chasing the following cycle of implications:
\[\begin{tikzcd}
	{\text{acc. rel. }\mathcal I} && {\mathscr B_{\mathcal I} \text{ f.g.}} \\
	{\text{acc. rel. }\mathcal H} && {\mathscr B_{\mathcal H} \text{ f.g.}}
	\arrow["{(\ref{claim:inf-ended-case})}",  Rightarrow, from=1-1, to=1-3]
	\arrow["{(\ref{claim:I-fg-implies-H-fg})}", Rightarrow, from=1-3, to=2-3]
	\arrow["{(\ref{claim:acc-H-implies-acc-I})}", Rightarrow, from=2-1, to=1-1]
	\arrow["{(\ref{claim:easy-case})}", Rightarrow, from=2-3, to=2-1]
\end{tikzcd}\]


\section{Finitely generated groups}\label{sec:groups}

We now discuss how our definition of relative accessibility relates to the usual group-theoretic notion. Throughout this section, we let $G$ be a finitely generated group, and fix a choice of locally finite Cayley graph $\Gamma$. Let $\mathcal P$ be a finite collection of subgroups of $G$, and let
$$
\per = \{gH : g \in G, \ H \in \mathcal P\}
$$
denote the set of left cosets of elements of $\mathcal P$. Note that since $\mathcal P$ is finite, this is a thin, $G$-invariant peripheral system of $\Gamma$. 

\subsection{Algebraic relative accessibility}

We will need to briefly establish some terminology relating to actions of $G$ upon trees. Our conventions will mostly follow those set out in \cite{jsj}. 

Let $T$ be a simplicial tree. We say that $T$ is a \emph{$G$-tree} if $G$ acts on $T$ by automorphisms. We say that this action is 
\emph{without inversions} if, given any $g \in G$, if $g$ fixes an edge then it fixes the two endpoints of said edge. Note that this condition is fairly inconsequential, as it can always be assured by simply passing to the barycentric subdivision of $T$.
Call $T$ \emph{non-trivial} if there the action has no global fixed point, and \emph{cocompact} if $T/G$ is compact. 
A $G$-tree is \emph{minimal} if there does not exist a proper $G$-invariant subtree. It is a standard fact that, since $G$ is finitely generated, every non-trivial $G$-tree contains a unique minimal $G$-invariant subtree, and the action on this minimal subtree is cocompact. 
If $T$ is a minimal $G$-tree without inversions, then we call $T$ a \emph{finite splitting} of $G$. 
Given two $G$-trees $T_1$, $T_2$, we say that $T_1$ \emph{dominates} $T_2$ if there exists a $G$-equivariant map $f : V(T_1) \to V(T_2)$. We will call such an $f$ a \emph{$G$-map}, and often abuse notation by writing $f : T_1 \to T_2$. 
We say that a subgroup $H \leq G$ acts \emph{elliptically} on $T$ if it fixes a vertex.
It is easy to see that $T_1$ dominates $T_2$ if and only if every vertex stabiliser of $T_1$ acts elliptically on $T_2$. Also, domination is not affected by passing to subdivisions, or replacing the $T_i$ with $G$-invariant subtrees.

A $G$-tree $T$ is said to be \emph{$\mathcal P$-elliptic} if for all $H \in \mathcal P$, we have that $H$ acts elliptically on $T$. 
The group is said to be \emph{one-ended relative to $\mathcal P$} if every finite splitting of $G$ is not $\mathcal P$-elliptic. 
The splitting $T$ is said to be \emph{terminal} if $T$ dominates any other $\mathcal P$-elliptic finite splitting $T'$. We now have the following definition.

\begin{definition}[Algebraic relative accessibility]
    If there exists a terminal $\mathcal P$-elliptic finite splitting of $G$, then we say that $G$ is \emph{accessible relative to $\mathcal P$}. 
\end{definition}

\subsection{Structure trees of nested sets}

We now recount the construction of the \emph{structure tree} of a nested subset of $\br(\Gamma)$. This construction is due to Dunwoody \cite{dunwoody1979accessibility}, though we will follow the more modern conventions set out by Roller \cite{roller-thesis}.

The following notation will be convenient. Given a tree $T$, write 
$$
\mathscr E(T) = \{b \in \br(T) : |\delta b| = 1\}.
$$
That is, $\mathscr E(T)$ consists of those single-edge cuts. Note that $\mathscr E(T)$ additively generates $\br(T)$.

Let $\Gamma$ be a connected $G$-graph. 
Let $E \subset \br(\Gamma)$ be a symmetric, $G$-invariant, $G$-finite subset. An \emph{ultrafilter} on $E$ is a subset $U \subset E$ such that:
\begin{enumerate}
    \item For all $b \in E$, exactly one of $b$ or $b^\ast$ lies in $U$.

    \item If $b \subset b'$ and $b \in U$, then $b' \in U$. 
\end{enumerate}
We say that $U$ is \emph{well-founded} if for all $b \in U$, the set $\{a \subset b : a \in U\}$ is finite.
Form a tree $T = T(E)$, called the \emph{structure tree} of $E$, as follows:
\begin{enumerate}
    \item $V(T) = \{\text{well-founded ultrafilters on $E$}\}$,

    \item Connect $U_1$, $U_2$ by an edge if and only if $|U_1 + U_2| = 2$. 
\end{enumerate}
This is a indeed a simplicial tree, upon which $G$ acts with finitely many orbits. There is a  canonical equivariant bijection between $E$ and $\mathscr E(T)$, which extends to an isomorphism of $G$-modules between $\br(T)$ and the submodule of $\br(\Gamma)$ generated by $E$, and in particular a canonical embedding $\br(T) \into \br(\Gamma)$.

The following observation is straightforward.

\begin{proposition}\label{prop:structure-tree-elliptic}
    If $E \subset \br_\per(\Gamma)$ is nested, $G$-invariant, and $G$-finite, then the structure tree $T = T(E)$ is $\mathcal P$-elliptic. 
\end{proposition}

\begin{proof}
    Suppose $T$ is not $\mathcal P$-elliptic, so there exists some $H \in \mathcal P$ which acts on $T$ with unbounded orbits.
    If $H$ contains an element acting loxodromically on $T$, then it is easy to see that some $b \in E$ must satisfy that both $b \cap H$ and $b^\ast \cap H$ are infinite. Thus, $H$ acts parabolically on $T$, fixing a unique end $\omega \in \Omega(T)$. However, this will contradict the fact that $T$ has finite edge stabilisers. Indeed, let $P$ be a geodesic in $T$, then for all $g \in H$, there must exist a tail $P' \subset P$ such that $gP' = P'$, lest $g$ be acting loxodromically. From here, it follows quickly that $T$ contains infinite edge stabilisers. 
\end{proof}

\subsection{Finite splittings and cuts}

We now discuss how finite splittings of $G$ relate to $\br(\Gamma)$. 
First, we have the following lemma. 

\begin{lemma}\label{lem:br-hom-well-defined}
    Let $\Delta$ be some connected, cocompact $G$-graph with finite edge stabilisers, and $T$ a $G$-tree with finite edge stabilisers. Let $f : V(\Delta) \to V(T)$ be a $G$-equivariant map. Then the map
    $$
    F: \br(T) \to \br(\Delta), \ \ F(b) = f^{-1}(b)
    $$
    is a well-defined, monotone, complement-respecting, $G$-equivariant ring homomorphism.
    Moreover, if $T$ is minimal then we have that 
    $$
    F^{-1}(\mathscr F(\Delta)) \subset \mathscr F(T).
    $$
\end{lemma}

\begin{proof}
    Since preimages distribute over symmetric difference and intersection, it is easy to see that $F$ is $G$-equivariant homomorphism of rings $F: \br(T) \to \mathscr P(\Delta)$. It is also clear that $F$ is monotone and complement-respecting. We must verify that $F(b)$ has finite coboundary for every $b \in \br(T)$. 
    It is sufficient to consider the case where $b \in \mathscr E(T)$, say $\delta b = \{e\}$. 
    Write $c = F(b)$, and assume that $\delta c$ is infinite. 
    Let $e_1, e_2, \ldots$ denote the edges in $\delta c$. Let $u_i$, $v_i$ denote the endpoints of $e_i$, say $u_i \in c$ and $v_i \in c^\ast$. 
    This means that $f(u_i) \in b$, and $f(v_i) \in b^\ast$. 
    Pick orbit representatives $e'_1, \ldots, e_k,' \in E(\Delta)$. 
    For each $i > 0$, fix $g_i \in G$ such that $e_i = g_ie_\ell'$ for some (unique) $e_\ell'$. Consider the geodesic $P_i = [f(u_i), f(v_i)]$, which must contain the edge $e$. Note that $\bigcup_i g_i^{-1}P_i$ is finite, since the endpoints of the $g_i^{-1}P_i$ are exactly the images of the endpoints of the $e'_\ell$, of which there are finitely many. 
    This induces a finite partition  on $\{g_i\}$, where $g_i$, $g_j$ lie in the same part if and only if $g_i^{-1} e = g_j^{-1} e$. Let $\pi$ be a part of this partition. Fix $g \in \pi$, and let $h \in \pi$ be arbitrary. Then 
    $$
    g^{-1}e = h^{-1}e \implies hg^{-1} \in G_e \implies h \in G_e g,
    $$ 
    Since $G_e$ is finite, we deduce that $\pi \subset G_eg$ is finite. In particular, this implies that $\{g_i\}$ is finite and so $\delta c$ is finite. This concludes the first part of this lemma.

    Suppose now that $T$ is minimal.
    Let $b \in \br(T)$, and suppose $F(b) \in \mathscr F(\Delta)$. Without loss of generality, let's assume that $F(b)$ is finite (as opposed to cofinite), since $F$ respects complements. 
    This means that $f^{-1}(b)$ is finite, and so $b \cap \operatorname{Im}(f)$ is finite. Since $T$ is cocompact, the image of $f$ coarsely covers $V(T)$; that is, there exists $r > 0$ such that $B(\operatorname{Im}(f),r)$ covers $T$. Since $\delta b$ is finite, $T[b]$ has finitely many connected components, and each of these components has finite coboundary in $T$. If one of these components $U$ were infinite, since $T$ is minimal and thus leafless, we must have that $U$ contains vertices arbitrarily far from $\delta U$. But then $U$ must intersect the image of $f$ infinitely often, which is a contradiction. In particular, we conclude that $F^{-1}(\mathscr F(\Delta)) \subset \mathscr F(T)$. 
\end{proof}

Let $T$ be a cocompact $G$-tree with finite edge stabilisers. By fixing a basepoint $x_0 \in V(T)$, and consider the orbit map 
    $$
    \varphi_{x_0} : V(\Gamma) \mapsto V(T), \ \ \ g  \mapsto g x_0. 
    $$
    This map is $G$-equivariant, and well-defined, since the action of $G$ upon itself is free. Pulling back, this induces a map
    $$
    \Phi_{x_0} : \br(T) \to \br(\Gamma), \ \ \ \Phi_{x_0}(b) = \varphi_{x_0}^{-1}(b).
    $$
    By Lemma~\ref{lem:br-hom-well-defined}, this is a well-defined homomorphism of $G$-modules. 

\begin{proposition}\label{prop:map-from-br-tree-to-br-gamma}
    Let $T$ be a cocompact $G$-tree with finite edge stabilisers. Then the following hold:
    \begin{enumerate}
        \item\label{itm:basic-1} If $T$ is $\mathcal P$-elliptic then $\operatorname{Im}(\Phi) \subset \br_\per(\Gamma)$. 

        \item\label{itm:basic-2} For all $x_0, y_0 \in V(T)$, $b \in \br(T)$, we have $\Phi_{x_0}(b) + \Phi_{y_0}(b)$ is finite.


        \item\label{itm:basic-4} Suppose that $T = T(E)$ is the structure tree of a nested, $G$-invariant, $G$-finite subset $E \subset \br(\Gamma)$, then there exists $x_0 \in V(T)$ such that $\Phi_{x_0}$ is precisely the canonical embedding $\br(T) \into \br(\Gamma)$. 
    \end{enumerate}
    
\end{proposition}

\begin{proof}
    We first prove (\ref{itm:basic-1}). Fix $x_0 \in V(T)$ and write $\Phi = \Phi_{x_0}$. 
    Suppose now that $T$ is $\mathcal P$-elliptic, and let $b \in \operatorname{Im}(\Phi)$. Suppose some $H \in \per$ has infinite intersection with both $b$ and $b^\ast$. Up to translating, we may assume that $H$ is one of the subgroups in $\mathcal P$. We necessarily have that $|\Lambda(H)| > 1$, as it intersects both $\Lambda(b)$ and $\Lambda(b^\ast)$. Then, $H$ must contain a loxodromic element $g$ with one limit point in $\Lambda(b)$, and the other in $\Lambda(b^\ast)$. Indeed, $\Omega(\Gamma)$ is bipartitioned into clopen sets $\Lambda(b) \sqcup \Lambda(b^\ast)$, with $\Lambda(H)$ intersecting both sides of this bipartition. By \cite[Thm.~2.6(i)]{Hamann}, there exists loxodromics $h_1, h_2 \in H$ such that 
    $$
    \Lambda(h_1) \cap \Lambda(b) \neq \emptyset, \ \ \ \Lambda(h_2) \cap \Lambda(b^\ast) \neq \emptyset. 
    $$
    If neither $h_i$ satisfy our requirements, then $\Lambda(h_1) \cap \Lambda(h_2) = \emptyset$, and so then we may apply \cite[Thm.~2.8]{Hamann} to find our desired loxodromic $g \in H$.
    But then, given this, it is clear that $g$ cannot fix a point on $T$, which contradicts the fact that $T$ is $\mathcal P$-elliptic.

    We now prove (\ref{itm:basic-2}). Let $x_0,y_0 \in V(T)$, and write $\Phi = \Phi_{x_0}$, $\Psi = \Phi_{y_0}$. Let $P$ denote the geodesic between $x_0$ and $y_0$, which contains finitely many edges. Let $b \in \br(T)$. It is sufficient to consider the case where $b \in \mathscr E(T)$, say $\delta b = \{e\}$. Consider 
    $$
    \Phi(b) + \Psi(b) = \varphi^{-1}(b) + \psi^{-1}(b).
    $$
    Suppose this were infinite, and enumerate the elements $g_1, g_2, \ldots$, and so on. Without loss of generality, let's say $g_i \in \varphi^{-1}(b) \setminus \psi^{-1}(b)$ for all $i > 0$, so $g_ix_0 \in b$ and $g_iy_0\in b^\ast$. This means that $g_iP$ contains the single edge $e$ in $\delta b$, and so $g^{-1}_ie$ lies in $P$ for all $i > 0$. Passing to a subsequence, we may assume that $g^{-1}_ie = g^{-1}_je$ for all $i,j > 0$. But then $g_i g^{-1}_je = e$ for all $i, j >0$. Since the $g_i$ are pairwise distinct, this implies the stabiliser of $e$ is infinite, which is a contradiction.


    Finally, we observe that (\ref{itm:basic-4}) holds. Suppose that $T = T(E)$ is a structure tree of a nested, $G$-invariant, $G$-finite, nested subset $E \subset \br(\Gamma)$. We must choose $x_0$ so that $\Phi_{x_0}$ recovers precisely the canonical embedding $\br(T) \into \br(\Gamma)$. There is a natural choice for such an $x_0$. Let
    $$
    U = \{b \in E: 1 \in b\}.
    $$
    It is easy to verify that $U$ is a well-founded ultrafilter on $E$, and in particular is a vertex of $T$. We then set $x_0 = U$. It is straightforward to verify that $\Phi_{x_0}$ recovers the canonical embedding $\br(T) \into \br(\Gamma)$. 
\end{proof}

\begin{lemma}\label{lem:g-map-gives-br-map}
    Let $T_1$, $T_2$ be cocompact $G$-trees with finite edge-stabilisers. Let $f : T_1 \to T_2$ be a $G$-map. Then this  induces a $G$-equivariant ring homomorphism
    $
    F : \br(T_2) \to \br(T_1).
    $
    Moreover, given $x_0 \in V(T_1)$, we have that 
    $$
    \Phi_{f(x_0)} = \Phi_{x_0} \circ F.
    $$
\end{lemma}

\begin{proof}
    The existence of $F$ is automatically ensured by Lemma~\ref{lem:br-hom-well-defined}, given by the formula
    $
    F(b) = f^{-1}(b)
    $. 
    Let $x_0 \in V(T_1)$ and write $y_0 = f(x_0)$. 
    Let $b \in \br(T_2)$. We have that $\Phi_{y_0}(b) = \varphi_{y_0}^{-1}(b) =: A$, where $\varphi_{y_0}(g) = gy_0$.
    We also have that 
    $$
    \Phi_{x_0} \circ F(b) = \varphi^{-1}_{x_0}(F^{-1}(b)) =: B.
    $$
    Given $g \in V(\Gamma)$, we have that
    $$
    g \in B \Longleftrightarrow f \circ \varphi_{x_0}(g) \in b \Longleftrightarrow f(gx_0) \in b \Longleftrightarrow g y_0 \in b \Longleftrightarrow g \in A.
    $$
    It follows that $\Phi_{f(x_0)} = \Phi_{x_0} \circ F$. 
\end{proof}


    





Conversely, we have the following. 

\begin{lemma}\label{lem:br-map-gives-gmap}
    Let $T_1$, $T_2$ be cocompact $G$-trees with finite edge-stabilisers. Suppose there is a monotone, complement-respecting, $G$-equivariant homomorphism of rings 
    $$
    \Phi : \br(T_2) \to \br(T_1)
    $$
    such that
    $
    \ker (\Phi) \cap \mathscr E(T_2) = \emptyset
    $. 
    Then $T_1$ dominates $T_2$.  
\end{lemma}

\begin{proof}
    Fix $v \in V(T_1)$, and let $G_v \leq G$ denote its stabiliser. It is sufficient to show that $G_v$ acts elliptically on $T_2$. Consider the inclusion $\iota : \mathscr E(T_2) \into \br(T_2)$, and let $\Psi = \Phi \circ \iota$. Consider 
    $
    U_1 = \{b \in \operatorname{Im}(\Psi) : v \in b\}
    $. 
    This is an ultrafilter on $\operatorname{Im}(\Psi)$. Let 
    $$
    U_2 = \Psi^{-1}(U_1).
    $$
    It is easy to verify that $U_2$ is an ultrafilter on $\mathscr E(T_2)$, since $\Phi$ is monotone and complement preserving. We claim that $U_2$ is well-founded. To see this, we will need the following. 

    \begin{claim}\label{claim:dcc}
        Suppose $b_1 \supset b_2 \supset \ldots$ is a strictly descending chain in $\mathscr E(T_2)$. Then
        $$
        \bigcap_i \Psi(b_i) = \emptyset.
        $$
    \end{claim}

    \begin{proof}
        Suppose $v \in \bigcap_i \Psi(b_i)$. Note that $(\Psi(b_i))_i$ is also a descending chain in $\br(T_1)$, since $\Phi$ is monotone. Passing to a subsequence we may assume the $b_i$ are orbit-equivalent, since $T_2$ is  cocompact. 
        For all $i > 0$, fix $g_i \in G$ such that $g_i b_i = b_1$. Note that all the $g_i$ are pairwise distinct. 
        Since $\ker (\Phi) \cap \mathscr E(T_2) = \emptyset$, we have that $\delta \Psi(b_i)$ is non-empty for all $i > 0$. Let $u \in  \Psi(b_1)^\ast$, and consider the geodesic $P$ connecting $u$ to $v$, which contains finitely many edges. Passing to another subsequence, we may assume without loss of generality that every $\delta \Psi(b_i)$ contains some fixed edge of $P$, say $e$.  
        Consider $g_i e$, which lies in $\delta \Psi(b_1)$. This is a finite collection of edges, so passing to yet another subsequence we may assume that $g_ie = g_je$ for all $i, j > 0$. But then $g_j^{-1}g_i e = e$ for all $i, j > 0$. Since the $g_i$ are pairwise distinct, this implies the stabiliser of $e$ is infinite, which is a contradiction. The claim follows. 
    \end{proof}

    Suppose now that $U_2$ is not well-founded. Then there exists a strictly descending chain $b_1 \supset b_2 \supset \ldots$ in $U_2$. But then $\bigcap_i \Psi(b_i) = \emptyset$, by Claim~\ref{claim:dcc}. Now, $\Psi(b_i) \in U_1$ for all $i> 0$, by definition, but then we must have that $v \not\in \Psi(b_i)$ for some $i$. This contradicts the construction of $U_1$, and so $U_2$ must be well-founded. 
    
    Note that $\operatorname{Im}(\Psi)$ is $G$-invariant, and so $G_v$ preserves $U_1$ and thus also preserves $U_2$. But $U_2$, being a well-founded ultrafilter on $\mathscr E(T_2)$, points to a unique vertex in $T_2$, which $G_v$ therefore fixes. It follows that $T_1$ dominates $T_2$. 
\end{proof}

\subsection{Proof of Theorem~\ref{thm:alg-vs-graph-relacc}} We are now able to prove that algebraic and graph-theoretic relative accessibility coincide in the case of finitely generated groups.

\algaccchar*

\begin{proof}
    First, let suppose that $\Gamma$ is accessible relative to $\per$. By Theorem~\ref{thm:fg-char}, we have that $\br_\per(\Gamma)$ admits a nested, $G$-invariant, $G$-finite additive generating set $E$. Let $T = T(E)$ be the structure tree associated to $E$. 
    Since every $b \in E = E(T)$ has finite coboundary and the action of $G$ on $\Gamma$ is free, we have that the stabiliser of every $b \in E$ is finite. We have by Proposition~\ref{prop:structure-tree-elliptic} that $T$ is $\mathcal P$-elliptic. Note there is a canonical isomorphism $\Phi : \br(T) \cong \br_\per(\Gamma)$, dual to some equivariant map $\varphi: \Gamma \to T$ by Proposition~\ref{prop:map-from-br-tree-to-br-gamma}(\ref{itm:basic-4}).
    
    Let $T'$ be some arbitrary minimal, $\mathcal P$-elliptic $G$-tree with finite edge stabilisers. By Lemma~\ref{lem:br-hom-well-defined} and Proposition~\ref{prop:map-from-br-tree-to-br-gamma}(\ref{itm:basic-1}), there exists a monotone, complement-respecting $G$-equivariant homomorphism 
    $$
    \Psi : \br(T') \to \br_\per(\Gamma)
    $$
    such that $\Psi^{-1}(\mathscr F(\Gamma)) \subset \mathscr F(T')$. 
    Consider the composition $\Theta := \Phi^{-1} \circ \Psi$, and note that $\ker(\Theta) \cap \mathscr E(T') = \emptyset$. Indeed, let $b \in \mathscr E(T')$, and suppose $\Theta(b) = \Phi^{-1} \circ \Psi(b) = \emptyset$. Since $\Phi$ is an isomorphism, this means that $\Psi(b) = \emptyset$, but then $b \in \mathscr F(T')$. Since $|\delta b| = 1$, this contradicts the fact that $T$ is minimal and thus leafless. 
    By Lemma~\ref{lem:br-map-gives-gmap}, we have that there is a $G$-map $T \to T'$. Since $T'$ was arbitrary, it follows that $T$ is a terminal $\mathcal P$-elliptic finite splitting (after possibly passing to a minimal $G$-invariant subtree of $T$), and so $G$ is accessible relative to $\mathcal P$. 

    Conversely, suppose that $T$ is a terminal $\mathcal P$-elliptic finite splitting. We will show that $\br_\per(\Gamma)$ is a finitely generated $G$-module. 
    By Theorem~\ref{thm:nested-gen-set-rel}, we have that $\br_\per(\Gamma)$ admits a nested, $G$-invariant additive generating set $E$. 
    Express $E$ as an ascending union
    $$
    E_1 \subset E_2 \subset \ldots \subset \bigcup_iE_i = E,
    $$
    where each $E_i$ is nested, $G$-invariant, and $G$-finite. Let $T_i$ denote the structure tree of $E_i$, which is a $G$-tree with finite edge stabilisers. 
    Let $M_i$ denote the submodule of $\br_\per(\Gamma)$ generated by $E_i$, so we have  natural isomorphisms of $G$-modules $\br(T_i) \cong M_i$.
    By Proposition~\ref{prop:map-from-br-tree-to-br-gamma}(\ref{itm:basic-4}), fix $z_i \in V(T_i)$ such that $\Phi_{z_i} : \br(T_i) \to \br_\per(\Gamma)$ recovers the natural inclusion $M_i \into \br_\per(\Gamma)$. 
    By assumption, there exists a domination map $f_i : T \to T_i$. Fix $x_0 \in V(T)$, and let $y_i = f_i(x_0)$. By Lemma~\ref{lem:g-map-gives-br-map}, we have that $f_i$ induces a homomorphism $F_i : \br(T_i) \to \br(T)$ such that
    $$
    \Phi_{y_i} = \Phi_{x_0} \circ F_i.
    $$
    Let $M = \operatorname{Im}(\Phi_{x_0})$.
    For every $b \in \br_\per(\Gamma)$, there exists $i > 0$ such that $b \in M_i$. In particular, there exists $c \in \br(T_i)$ such that $\Phi_{z_i}(c) = b$. By Proposition~\ref{prop:map-from-br-tree-to-br-gamma}(\ref{itm:basic-2}), we have that 
    $
    \Phi_{y_i}(c) + \Phi_{z_i}(c)
    $
    is finite. In other words, there exists $b' \in M$ such that $b + b'$ is finite. 
    Since $b$ was arbitrary, we have that 
    $$
    \br_\per(\Gamma) = M + \mathscr F(\Gamma).
    $$
    In particular, $\br_\per(\Gamma)$ is $G$-finitely generated, and so by Theorem~\ref{thm:fg-char} we conclude that $\Gamma$ is accessible relative to $\per$. 
\end{proof}

\section{Cycles in the cone-off}

In this section we will apply the cone-off construction to deduce a relative variant of Hamann's accessibility theorem \cite{hamann2018accessibility}.

\subsection{The cycle space of a graph} First, we recall some basic definitions and terminology related to the cycle space of a graph. 
Given a graph $\Gamma$, write 
$$
\Z_2^{E(\Gamma)} = \prod_{E(\Gamma)}\Z_2,
$$
equipped with the product topology. By Tychonoff's theorem, this is a compact topological abelian group. The \emph{cycle space} $\cyc(\Gamma)$ is the subgroup of $\Z_2^{E(\Gamma)}$ consisting of all finite Eulerian subgraphs, or equivalently the subgroup generated by all simple cycles. Note that $\cyc(\Gamma)$ can be identified with the first simplicial homology group of $\Gamma$ with $\Z_2$-coefficients.

\subsection{Topological generation}\label{sec:top-gen}

We wish to work with a slightly weaker notion of generating set for the cycle space, taking advantage of the fact that $\Z_2^{E(\Gamma)}$ is a compact topological abelian group. This will be useful for applications. 
We introduce the following terminology.

\begin{definition}[Topological generation]\label{def:topological-finite-gen}
    Let $\Gamma$ be a graph. We say that a subset $S \subset \cyc(\Gamma)$ is a \emph{topological generating set} if 
    $$
    \cyc(\Gamma) \subset \overline {\langle  S \rangle}.
    $$
    If $\Gamma$ is a $G$-graph, then we say that $\cyc(\Gamma)$ is \emph{topologically $G$-finitely generated} if there exists a $G$-invariant, $G$-finite topological generating set $S \subset \cyc(\Gamma)$.
\end{definition}

To distinguish topological generation from the case where $\langle S \rangle = A$, we may refer to the latter as \emph{algebraic generation}. 
We will be interested in the case where $A$ is the cycle space of $\Gamma$. 

\begin{remark}
    When $\Gamma$ is locally finite, the closure  $\overline {\cyc(\Gamma)}$ of the cycle space $\cyc(\Gamma)$ in $\Z_2^{E(\Gamma)}$  can be identified with the \emph{topological cycle space} of $\Gamma$, as defined by Diestel and K\"uhn \cite{diestel2004infinite}.
\end{remark}

\begin{remark}\label{rem:derivations}
    In the context of finitely generated groups, topological generation of the cycle space has already appeared in the literature under a different guise. Let $G$ be a finitely generated group and $\Gamma$ a Cayley graph of $G$, and suppose $\cyc(\Gamma)$ is topologically $G$-finitely generated. It is not too hard to see that this is equivalent to saying that $G$ admits a free and cocompact action on a simplicial 2-complex $K$ such that the map
    $$
    H_1(K;\Z_2 ) \to H_1^\infty(K;\Z_2 )
    $$
    vanishes, where $H_\ast^\infty$ denotes \emph{locally finite homology}; see \cite[\S11]{geoghegan2008topological} for a definition. Since $\Z_2 $ is a field, taking the dual we see that this is equivalent to the natural map
    $$
    H^1_{\mathrm{c}}(K;\Z_2 ) \to H^1(K;\Z_2 )
    $$
    vanishing, where $H^\ast_{\mathrm{c}}$ denotes \emph{compactly supported cohomology}. Groups with such an action are considered in \cite{groves1991remarks} by Groves and Swarup. Such groups are characterised algebraically amongst finitely generated groups as those which are \emph{finitely defined by derivations over $\Z_2 $}. This property is strictly weaker than being $\mathrm{FP}_2$ over $\Z_2 $ (sometimes known as being \emph{almost finitely presented}), which is well-known to be equivalent to the cycle space $\cyc(\Gamma)$ being algebraically $G$-finitely generated.
\end{remark}

Note that topological finite generation is strictly weaker than asking $A$ itself to be finitely generated as a $G$-module, as can be seen in the following examples.

\begin{example}
    We give two examples:
    \begin{enumerate}
        \item Consider the planar graph $\Gamma$ depicted in Figure~\ref{fig:topfgnotfg}. The cycle space of $\Gamma$ is topologically generated by cycles of length 4, but any algebraic generating set of $\cyc(\Gamma)$ must contain arbitrarily long cycles.

        \item For a quasi-transitive example, consider any locally finite Cayley graph $\Gamma$ of the group $G = \Z \wr \Z$. We have that $\cyc(\Gamma)$ is topologically $G$-finitely generated, but not algebraically $G$-finitely generated, since $G$ is not $\mathrm{FP}_2$ over $\Z_2 $ \cite{groves1991remarks}.
    \end{enumerate}
\end{example}

\begin{figure}
     \centering
    \input{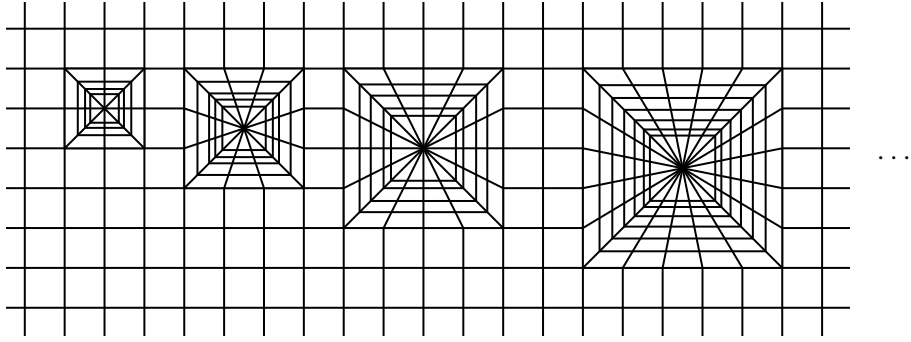}
    \caption{The cycle space of this planar graph is topologically generated by cycles of length 4, but any algebraic generating set must contain arbitrarily long cycles.}
    \label{fig:topfgnotfg}
\end{figure}

\subsection{Cycles and relative accessibility}

We now discuss how the cycle space influences relative accessibility. First, we state Hamann's theorem \cite{hamann2018accessibility}, which is a combinatorial formulation of Dunwoody's accessibility theorem \cite{dunwoody1985accessibility}. 

\begin{theorem}[Hamann]\label{thm:hamann-top}
    Let $\Gamma$ be a connected, 2-edge-connected $G$-graph. Suppose that $\cyc(\Gamma)$ is topologically $G$-finitely generated. Then $\br(\Gamma)$ is $G$-finitely generated. 
\end{theorem}

Note that our statement is in terms of \emph{topological} $G$-finite generation, whereas the original statement in \cite{hamann2018accessibility} only makes reference to algebraic $G$-finite generation. This upgrade only requires minor modifications to Hamann's proof, which we discuss in Appendix~\ref{app:hamann}. 

First, we record the following proposition, which allows us to weaken our assumptions on the peripheral system somewhat. 

\begin{proposition}\label{prop:thin-to-tame-cycles}
    Let $\Gamma$ be a connected, locally finite, quasi-transitive $G$-graph, and $\per$ a thin, $G$-invariant peripheral system. 
    Suppose $\cyc(\widehat \Gamma_\per)$ is topologically $G$-finitely generated. Then $\per$ is tame. 
\end{proposition}

\begin{proof}

    Suppose $\per$ is not tame. By Lemma~\ref{lem:unbounded-components-tame}, we have that some $H \in \per$ is infinite, but does not have unbounded coarse components. That is, for every $r > 0$ we have that every connected component of $\Gamma[B(H,r)]$ is  finite. However, it is then easy to form cycles which cannot be filled by short cycles, even topologically. Indeed, fix some $G$-finite subset $S \subset \cyc(\widehat \Gamma_\per)$, and let $r > 0$ be much larger than the size of the support of any cycle in $S$. Let $U, W \subset \Gamma[B(H,r)]$ be two connected components. Let $x \in H \cap U$, $y \in H \cap W$. Consider a path $P$ through $\Gamma$ connecting $x$ to $y$, and complete this to a cycle  $C$ in $\widehat \Gamma_\per$ by adding the two edges joining $x$ and $y$ to the cone-vertex $v_H$. Adding any elements of $S$ to $C$, a parity argument tells us that there will always be two edges leftover which are joining $v_H$ to $U$ and $W$, respectively. See Figure~\ref{fig:thin-to-tame-cycles} for an illustration.
\end{proof}

\begin{remark}
    The proof of Proposition~\ref{prop:thin-to-tame-cycles} also applies to algebraic $G$-finite generation. In this case, the conclusion is stronger: the peripherals must be coarsely connected. This can be applied to give another proof of a result of Osin, which states that if a finitely generated group is finitely presented relative to a collection $\mathcal P$ of peripheral subgroups, then $\mathcal P$ is finite and each $H \in \mathcal P$ is finitely generated \cite[Thm.~1.1]{osin2004relhyp}. 
\end{remark}

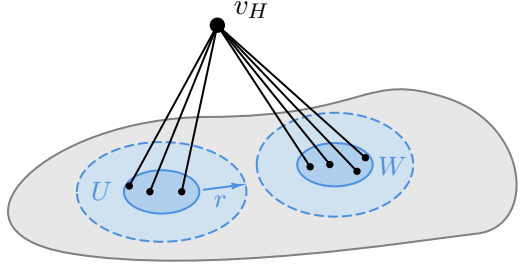
\begin{figure}
     \centering
    \tikzset{every picture/.style={line width=0.75pt}} 

\begin{tikzpicture}[x=0.75pt,y=0.75pt,yscale=-1,xscale=1]

\draw  [color={rgb, 255:red, 128; green, 128; blue, 128 }  ,draw opacity=1 ][fill={rgb, 255:red, 231; green, 231; blue, 231 }  ,fill opacity=1 ] (254.19,103.1) .. controls (335.52,102.73) and (332.23,79.81) .. (374.07,93.63) .. controls (415.91,107.46) and (419.2,158.22) .. (389.11,161.63) .. controls (359.02,165.04) and (250.9,183.04) .. (183.67,171.29) .. controls (116.45,159.55) and (172.86,103.48) .. (254.19,103.1) -- cycle ;
\draw  [color={rgb, 255:red, 74; green, 144; blue, 226 }  ,draw opacity=1 ][fill={rgb, 255:red, 208; green, 225; blue, 240 }  ,fill opacity=1 ][dash pattern={on 3.75pt off 1.5pt}] (188.38,140.7) .. controls (188.38,126.89) and (207.42,115.7) .. (230.92,115.7) .. controls (254.42,115.7) and (273.46,126.89) .. (273.46,140.7) .. controls (273.46,154.51) and (254.42,165.71) .. (230.92,165.71) .. controls (207.42,165.71) and (188.38,154.51) .. (188.38,140.7) -- cycle ;
\draw  [color={rgb, 255:red, 74; green, 144; blue, 226 }  ,draw opacity=1 ][fill={rgb, 255:red, 208; green, 225; blue, 240 }  ,fill opacity=1 ][dash pattern={on 3.75pt off 1.5pt}] (278.64,127.06) .. controls (278.64,112.68) and (296.21,101.02) .. (317.89,101.02) .. controls (339.57,101.02) and (357.14,112.68) .. (357.14,127.06) .. controls (357.14,141.45) and (339.57,153.11) .. (317.89,153.11) .. controls (296.21,153.11) and (278.64,141.45) .. (278.64,127.06) -- cycle ;
\draw  [color={rgb, 255:red, 74; green, 144; blue, 226 }  ,draw opacity=1 ][fill={rgb, 255:red, 176; green, 207; blue, 237 }  ,fill opacity=1 ] (212,140.7) .. controls (212,134.71) and (220.47,129.86) .. (230.92,129.86) .. controls (241.37,129.86) and (249.84,134.71) .. (249.84,140.7) .. controls (249.84,146.69) and (241.37,151.55) .. (230.92,151.55) .. controls (220.47,151.55) and (212,146.69) .. (212,140.7) -- cycle ;
\draw  [color={rgb, 255:red, 74; green, 144; blue, 226 }  ,draw opacity=1 ][fill={rgb, 255:red, 176; green, 207; blue, 237 }  ,fill opacity=1 ] (298.97,127.06) .. controls (298.97,121.08) and (307.44,116.22) .. (317.89,116.22) .. controls (328.34,116.22) and (336.81,121.08) .. (336.81,127.06) .. controls (336.81,133.05) and (328.34,137.91) .. (317.89,137.91) .. controls (307.44,137.91) and (298.97,133.05) .. (298.97,127.06) -- cycle ;
\draw    (258.89,57.08) ;
\draw [shift={(258.89,57.08)}, rotate = 0] [color={rgb, 255:red, 0; green, 0; blue, 0 }  ][fill={rgb, 255:red, 0; green, 0; blue, 0 }  ][line width=0.75]      (0, 0) circle [x radius= 3.35, y radius= 3.35]   ;
\draw    (258.89,57.08) -- (305.43,128.11) ;
\draw [shift={(305.43,128.11)}, rotate = 56.77] [color={rgb, 255:red, 0; green, 0; blue, 0 }  ][fill={rgb, 255:red, 0; green, 0; blue, 0 }  ][line width=0.75]      (0, 0) circle [x radius= 1.34, y radius= 1.34]   ;
\draw [shift={(258.89,57.08)}, rotate = 56.77] [color={rgb, 255:red, 0; green, 0; blue, 0 }  ][fill={rgb, 255:red, 0; green, 0; blue, 0 }  ][line width=0.75]      (0, 0) circle [x radius= 1.34, y radius= 1.34]   ;
\draw    (258.89,57.08) -- (241.03,140.61) ;
\draw [shift={(241.03,140.61)}, rotate = 102.07] [color={rgb, 255:red, 0; green, 0; blue, 0 }  ][fill={rgb, 255:red, 0; green, 0; blue, 0 }  ][line width=0.75]      (0, 0) circle [x radius= 1.34, y radius= 1.34]   ;
\draw [shift={(258.89,57.08)}, rotate = 102.07] [color={rgb, 255:red, 0; green, 0; blue, 0 }  ][fill={rgb, 255:red, 0; green, 0; blue, 0 }  ][line width=0.75]      (0, 0) circle [x radius= 1.34, y radius= 1.34]   ;
\draw    (258.89,57.08) -- (225.04,140.61) ;
\draw [shift={(225.04,140.61)}, rotate = 112.06] [color={rgb, 255:red, 0; green, 0; blue, 0 }  ][fill={rgb, 255:red, 0; green, 0; blue, 0 }  ][line width=0.75]      (0, 0) circle [x radius= 1.34, y radius= 1.34]   ;
\draw [shift={(258.89,57.08)}, rotate = 112.06] [color={rgb, 255:red, 0; green, 0; blue, 0 }  ][fill={rgb, 255:red, 0; green, 0; blue, 0 }  ][line width=0.75]      (0, 0) circle [x radius= 1.34, y radius= 1.34]   ;
\draw    (258.89,57.08) -- (214.7,137.77) ;
\draw [shift={(214.7,137.77)}, rotate = 118.71] [color={rgb, 255:red, 0; green, 0; blue, 0 }  ][fill={rgb, 255:red, 0; green, 0; blue, 0 }  ][line width=0.75]      (0, 0) circle [x radius= 1.34, y radius= 1.34]   ;
\draw [shift={(258.89,57.08)}, rotate = 118.71] [color={rgb, 255:red, 0; green, 0; blue, 0 }  ][fill={rgb, 255:red, 0; green, 0; blue, 0 }  ][line width=0.75]      (0, 0) circle [x radius= 1.34, y radius= 1.34]   ;
\draw    (258.89,57.08) -- (333.17,123.56) ;
\draw [shift={(333.17,123.56)}, rotate = 41.83] [color={rgb, 255:red, 0; green, 0; blue, 0 }  ][fill={rgb, 255:red, 0; green, 0; blue, 0 }  ][line width=0.75]      (0, 0) circle [x radius= 1.34, y radius= 1.34]   ;
\draw [shift={(258.89,57.08)}, rotate = 41.83] [color={rgb, 255:red, 0; green, 0; blue, 0 }  ][fill={rgb, 255:red, 0; green, 0; blue, 0 }  ][line width=0.75]      (0, 0) circle [x radius= 1.34, y radius= 1.34]   ;
\draw    (258.89,57.08) -- (328.94,130.38) ;
\draw [shift={(328.94,130.38)}, rotate = 46.3] [color={rgb, 255:red, 0; green, 0; blue, 0 }  ][fill={rgb, 255:red, 0; green, 0; blue, 0 }  ][line width=0.75]      (0, 0) circle [x radius= 1.34, y radius= 1.34]   ;
\draw [shift={(258.89,57.08)}, rotate = 46.3] [color={rgb, 255:red, 0; green, 0; blue, 0 }  ][fill={rgb, 255:red, 0; green, 0; blue, 0 }  ][line width=0.75]      (0, 0) circle [x radius= 1.34, y radius= 1.34]   ;
\draw    (258.89,57.08) -- (315.3,126.97) ;
\draw [shift={(315.3,126.97)}, rotate = 51.09] [color={rgb, 255:red, 0; green, 0; blue, 0 }  ][fill={rgb, 255:red, 0; green, 0; blue, 0 }  ][line width=0.75]      (0, 0) circle [x radius= 1.34, y radius= 1.34]   ;
\draw [shift={(258.89,57.08)}, rotate = 51.09] [color={rgb, 255:red, 0; green, 0; blue, 0 }  ][fill={rgb, 255:red, 0; green, 0; blue, 0 }  ][line width=0.75]      (0, 0) circle [x radius= 1.34, y radius= 1.34]   ;
\draw [color={rgb, 255:red, 74; green, 144; blue, 226 }  ,draw opacity=1 ]   (251.67,139.67) -- (268.35,137.48) ;
\draw [shift={(270.33,137.22)}, rotate = 172.54] [color={rgb, 255:red, 74; green, 144; blue, 226 }  ,draw opacity=1 ][line width=0.75]    (4.37,-1.32) .. controls (2.78,-0.56) and (1.32,-0.12) .. (0,0) .. controls (1.32,0.12) and (2.78,0.56) .. (4.37,1.32)   ;

\draw (194.09,134.24) node [anchor=north west][inner sep=0.75pt]  [font=\small,color={rgb, 255:red, 74; green, 144; blue, 226 }  ,opacity=1 ]  {$U$};
\draw (338.07,121.79) node [anchor=north west][inner sep=0.75pt]  [font=\small,color={rgb, 255:red, 74; green, 144; blue, 226 }  ,opacity=1 ]  {$W$};
\draw (265.68,43.61) node [anchor=north west][inner sep=0.75pt]  [font=\normalsize]  {$v_{H}$};
\draw (255.74,140.79) node [anchor=north west][inner sep=0.75pt]  [font=\small,color={rgb, 255:red, 74; green, 144; blue, 226 }  ,opacity=1 ]  {$r$};

\end{tikzpicture}
    \caption{Illustration of the proof of Proposition~\ref{prop:thin-to-tame-cycles}.}
    \label{fig:thin-to-tame-cycles}
\end{figure}

We now deduce a relative variant of Theorem~\ref{thm:hamann-top}, which is the last main result of this paper. 

\relcycs*

\begin{proof}
    By Proposition~\ref{prop:thin-to-tame-cycles}, we have that $\per$ is tame. 
    First, note that we may assume without loss of generality that the cone-off $\widehat \Gamma_\per$ is 2-edge-connected as $G$ acts on $E(\widehat \Gamma_\per)$ with finitely many orbits, so single-edge cuts are inconsequential as there are only $G$-finitely many of them.
    
    By Theorem~\ref{thm:hamann-top}, $\br(\widehat \Gamma_\per)$ is $G$-finitely generated. But then $\br_\per(\Gamma)$ is also $G$-finitely generated by Proposition~\ref{prop:ring-isomorphism}, and so by Theorem~\ref{thm:fg-char} we have that $\Gamma$ is accessible relative to $\per$. 
\end{proof}

\bibliography{references}


\appendix

\section{Hamann's accessibility theorem and topological cycles}\label{app:hamann}

In this appendix we briefly discuss Hamann's accessibility theorem, and observe that the hypotheses can be weakened to relate to only requiring topological generation of the cycle space. We restate Hamann's theorem below, as it appears in \cite{hamann2018accessibility}. 

\begin{theorem}[Hamann]\label{thm:hamann}
    Let $\Gamma$ be a connected, 2-edge-connected $G$-graph. Suppose that $\cyc(\Gamma)$ is $G$-finitely generated. Then $\br(\Gamma)$ is $G$-finitely generated. 
\end{theorem}

We will weaken the hypotheses of this theorem, only requiring that $\cyc(\Gamma)$ is topologically $G$-finitely generated, in the sense of Definition~\ref{def:topological-finite-gen}. The proof, however, is essentially the same as Hamann's proof, only requiring minor alterations.
The key observation is that the full power of algebraic generation of the cycle space is never really employed, but instead one can make do applying the following result. 

\begin{proposition}\label{prop:top-gen-char}
    Let $\Gamma$ be a graph, $S \subset \cyc(\Gamma)$. If $S$ is a topological generating set of $\cyc(\Gamma)$, then the following holds:
    \begin{enumerate}
        \item[$(\dagger)$] For all $C \in \cyc(\Gamma)$ and all finite $U \subset E(\Gamma)$, there exists $C_1, \ldots, C_n \in S$ such that $C + \sum_i C_i$ is not supported on $U$.
    \end{enumerate}
    Conversely, if $\Gamma$ is countable and $S$ satisfies $(\dagger)$, then $S$ is a topological generating set of $\cyc(\Gamma)$.
\end{proposition}

\begin{proof}
    Suppose that $S$ is a topological generating set. Let $C \in \cyc(\Gamma)$ and $U \subset E(\Gamma)$ be arbitrary, with $U$ finite. By (1), there exists a sequence $(B_i)$ in $\langle S \rangle$ such that $B_i \to C$. Then there exists $N \geq 1$ such that for all $i \geq N$ we have that $B_i |_U = C |_U$. In particular, $B_i + C$ is not supported on $U$, and thus $(\dagger)$ holds.

    Now, assume that $(\dagger)$ holds and that $\Gamma$ is countable. Since $\Gamma$ is countable we may choose an enumeration $\{e_1, e_2\ldots\}$ of its edges. Let $E_i = \{e_1, \ldots, e_i\}$, so $E(\Gamma) = \bigcup_iE_i$.  Fix an arbitrary $C \in \cyc(\Gamma)$. For each $i \geq 1$, let $B_i \in \langle S \rangle$ be such that $B_i + C$ is not supported on $E_i$. It is now easy to verify that $B_i \to C$ in the product topology. Since $C$ was arbitrary, we have that $S$ is a topological generating set.
\end{proof}



With the above result in-hand, the rest of the proof is basically identical to Hamann's argument, with one application of Proposition~\ref{prop:top-gen-char} strategically inserted. We will reproduce the argument in its entirety below, for the sake of completeness. 

For the rest of this section, let $\Gamma$ be a fixed 2-edge-connected $G$-graph. By Theorem~\ref{thm:dicks-dunwoody}, we will also fix an increasing chain 
    $
     E _1 \subset  E _2 \subset \ldots
    $
    of nested, $G$-invariant subsets of $\br (\Gamma)$ consisting of tight elements such that $ E _n$ generates $\br_n(\Gamma)$ as a ring. Since $\Gamma$ is 2-edge-connected, $ E _1$ is empty.

Let us introduce some terminology. Two $b, b' \in \br(\Gamma)$ are said \textit{coincide} on a set $X \subset V(\Gamma)$ if $b \cap X = b' \cap X$. We say that $b$ \emph{partitions $X$ non-trivially} if both $b \cap X$ and $b^\ast \cap X$ are non-empty.
We restate the following technical lemmas from \cite{hamann2018accessibility}.

\begin{lemma}[{\cite[Lem.~2.3]{hamann2018accessibility}}]\label{lem:hamann-2.3}
    Let $X \subset V(\Gamma)$ be a non-empty set. Let $n \in \N$. Let $b \in  E _n$ partition $X$ non-trivially. Then the set
    $$
     E _{b, X} := \{b' \in  E _n : \text{$b'$ coincides with $b$ on $X$}\}
    $$
    forms a finite chain under inclusion.
\end{lemma}

\begin{lemma}[{\cite[Lem.~2.4]{hamann2018accessibility}}]\label{lem:hamann-2.4}
    Let $C \in \cyc(\Gamma)$ be a simple cycle and $b \in \br(\Gamma)$. Suppose that $C$ intersects $\delta b$ in exactly two edges $e, f \in E(\Gamma)$. Let $\mathcal A \subset \cyc(\Gamma)$ be a finite set of simple cycles such that $C = \sum_{A \in \mathcal A} A$. Then there exists an alternating sequence 
    $$
    e_1 A_1 e_2A_2 \ldots A_{n-1}e_n,
    $$
    with each $A_i \in \mathcal A$, $e_i, e_{i+1} \in \delta b \cap A_i$, $e_1 = e$ and $e_n = f$. 
\end{lemma}

We now prove the following. 

\begin{theorem}
    Let $\Gamma$ be a connected, 2-edge-connected $G$-graph. Suppose that $\cyc(\Gamma)$ is topologically $G$-finitely generated. Then $\br(\Gamma)$ is $G$-finitely generated. 
\end{theorem}

\begin{proof}
    Fix a topological $G$-finite generating set $S \subset \cyc (\Gamma)$. Let $\ell \in \N$ be the maximal number of edges supported by any cycle in $S$, and let $N$ denote the number of $G$-orbits of elements of $S$.  
    We will see that every $ E _n$ contains at most 
    $2^{1+\ell} N$ orbits, 
    and thus the chain of the $ E _n$ is eventually constant. 
    
    To this end, fix $n \in \N$. Let $C$ be any simple cycle. Let $X \subset V(C)$ be any non-empty, proper subset of $V(C)$. Write $Y = V(C) \setminus X$. Consider all those $b \in  E _n$ such that both $X = b \cap V(C)$ and $Y = b^\ast \cap V(C)$. By Lemma~\ref{lem:hamann-2.3}, these form a finite chain under inclusion, and so there are unique $\subset$-minimals and $\subset$-maximal elements. Letting $X$ vary, we see that there are at most $2^{|V(C)|}$ elements of $ E _n$ which induce a non-trivial bipartition of $V(C)$ and are maximal (or minimal) in the finite chain of $b \in  E _n$ which induce the same partition. If $ E _n$ contains strictly more than $2^{1+\ell} N$ orbits, we see that there must exist some $b_0 \in  E _n$ such that, 
    for any $C \in  S $, if $b$ partitions $V(C)$ non-trivially then it is neither the minimal nor maximal element of the chain $ E _{b,V(C)}$. Fix $b_0 \in  E _n$ with this property for the rest of this proof.

    \begin{claim}\label{claim:hamann-claim}
        Let $C_1, C_2 \in  S $. Let $b_1 \in  E _n$, and take $b_0$ as fixed above. Suppose that:
        \begin{enumerate}
            \item $C_1$ and $C_2$ share a common edge in $\delta b_0$, 

            \item $b_0$ and $b_1$ coincide on $V(C_1)$.

            \item $b_1$ is the immediate predecessor to $b_0$ in the finite chain $ E _{b_0,V(C_1)}$.
        \end{enumerate} 
        Then $b_0$ and $b_1$ coincide on $V(C_2)$, and $b_1$ is the immediate predecessor to $b_0$ in $ E _{b_0,V(C_2)}$. 
    \end{claim}

    \begin{proof}
        The hypotheses imply that there exists $e \in \delta b_0 \cap \delta b_1 \cap C_1$. Clearly then, both $b_0$ and $b_1$ partition each of $V(C_1)$ and $V(C_2)$ non-trivially. Note that any $b \in  E _n$ which satisfies $b_1 \subsetneq b \subsetneq b_0$ coincide with $b_0$ on $V(C_1)$. This contradicts the assumption that $b_1$ is the immediate predecessor to $b_0$ in $ E _{b_0,V(C_1)}$, and so no such $b$ exists. Thus, we need only show that $b_1$ and $b_0$ coincide on $V(C_2)$. 

        Let $u$, $v$ be the endpoints of $e$.
        Let $X = \{u,v\}$, so $b_0$ and $b_1$ coincide on $X$ and partition $X$ non-trivially. By Lemma~\ref{lem:hamann-2.3}, the set $ E _{b_0,X}$ is also a finite chain containing $b_0$ and $b_1$, and $b_1$ is necessarily the immediate predecessor of $b_0$. 
        Consider now $ E _{b_0,V(C_2)}$. By construction, $b_0$ has an immediate predecessor in this chain, say $b_2$. By symmetrical reasoning to earlier, $b_2$ must also be the immediate predecessor to  $b_0$ in $ E _{b_0,X}$. This implies that $b_2 = b_1$, and thus $b_1$ and $b_0$ coincide on $V(C_2)$. 
    \end{proof}

    We now inductively apply the above claim. 
    Let $D$ be any cycle which is partitioned non-trivially by $b_0$. This exists since $|\delta b_0| > 1$ and $b_0$ is tight. Let $b_1$ be the immediate predecessor of $b_0$ in $ E _{b_0,V(D)}$. Thus, there is some edge $e \in D \cap \delta b_0 \cap \delta b_1$. Let $f \in \delta b_0$ be any arbitrary and distinct from $e$. We will see that $f \in \delta b_1$. 
    
    Since $b_0$ is tight and $\delta b_0$ contains at least two edges, there exists a simple cycle $C$ such that $C$ contains exactly two edges from $\delta b_0$, and these two edges are exactly $e$ and $f$. We now apply Proposition~\ref{prop:top-gen-char}, and write $C$ as a sum $C = \sum_j D_j$ such that each $D_j$ is either an element of $S$ or is disjoint from $\delta b_0$.  
    We apply Lemma~\ref{lem:hamann-2.4} to this decomposition of $C$, and see that there exists an alternating sequence 
    $$
    e_1 C_1 e_2C_2e_3 \ldots C_{m-1}e_m,
    $$
    with each $C_i \in  S $, $e_i, e_{i+1} \in \delta b_0 \cap C_i$, $e_1 = e$ and $e_m = f$. We may assume without loss of generality that $C_1 = C$ by inserting $eC$ to the beginning of our sequence otherwise.
    By Claim~\ref{claim:hamann-claim}, we have that $b_1$ and $b_0$ also coincide on $C_2$. Inductively, we see that $b_1$ and $b_0$ coincide on all the $C_i$. Since $f$ lies on $C_{m-1}$, we deduce that $f \in \delta b_1$. Since $f$ was arbitrary, $\delta b_1 \supset \delta b_0$. But $b_0 \supset b_1$, and so we must have that $b_1 = b_0$. This contradicts the assumption that $b_0$ and $b_1$ are distinct. 
    
    Thus, the $ E _n$ eventually stabilise. Since each $ E _n$ is $G$-finite and $\bigcup_n  E _n$ generates $\br(\Gamma)$, it follows that $\br(\Gamma)$ is a finitely generated $G$-module. 
\end{proof}

As an immediate corollary, by combining this theorem with the discussion given in Remark~\ref{rem:derivations} we recover the following extension of Dunwoody's accessibility theorem \cite{dunwoody1985accessibility}, due to Groves and Swarup \cite[Thm.~1.2]{groves1991remarks}.

\begin{corollary}[Groves--Swarup]
    Let $G$ be a finitely generated group which is finitely defined by derivations over $\Z_2 $. Then $G$ is accessible.
\end{corollary}

\end{document}